\documentclass[a4paper]{amsart}

\usepackage{a4wide}

\newtheorem{theorem}{Theorem}[section]
\newtheorem{lemma}[theorem]{Lemma}

\newtheorem{proposition}[theorem]{Proposition}
\newtheorem{corollary}[theorem]{Corollary}

\theoremstyle{definition}
\newtheorem{definition}[theorem]{Definition}
\newtheorem{example}[theorem]{Example}

\theoremstyle{remark}
\newtheorem{remark}[theorem]{Remark}

\numberwithin{equation}{section}

\usepackage[normalem]{ulem}
\usepackage[usenames,dvipsnames,svgnames,table]{xcolor}
\usepackage{mathtools}
\usepackage{enumerate}
\usepackage{cancel}

\newcommand{\defeq}{\vcentcolon=}

\newcommand{\dt}{\Delta t}

\newcommand{\dx}{\Delta x}
\newcommand{\eps}{\ensuremath{\varepsilon}}

\newcommand{\R}{\mathbb{R}}
\newcommand{\A}{\mathcal{A}}
\newcommand{\II}{\mathcal{I}}
\newcommand{\diff}{\, \mathrm{d}}

\newcommand{\cI}{\mathcal{I}}

\def\e{\varepsilon}
\def\G{\mathcal{G}}
\def\N{\mathbb{N}}

\begin{document}

\title[Regularity and convergence for HJB equations in domains]{Some regularity and convergence results for parabolic Hamilton-Jacobi-Bellman equations in bounded domains}

\author{Athena Picarelli}
\address{Universit{\`a} degli studi di Verona, via Cantarane, 37129 Verona, Italy}
\email{athena.picarelli@univr.it}

\author{Christoph Reisinger}
\address{Mathematical Institute, University of Oxford, Woodstock Road, Oxford, OX2 6GG, United Kingdom}
\email{christoph.reisinger@maths.ox.ac.uk}

\author{Julen Rotaetxe Arto}
\address{Mathematical Institute, University of Oxford, Woodstock Road, Oxford, OX2 6GG, United Kingdom}
\email{julenrotaetxe@gmail.com}
\thanks{The third author was supported in part by the \emph{Programa de Formaci{\'o}n de Investigadores del DEUI del Gobierno Vasco}.}


\date{}



\begin{abstract}
We study the approximation of parabolic Hamilton-Jacobi-Bellman (HJB) equations in bounded domains with strong Dirichlet boundary conditions. We work under the assumption of the existence of a sufficiently regular barrier function for the problem to obtain well-posedness and regularity of a related switching system and the convergence of its components to the HJB equation. In particular, we show existence of a viscosity solution to the switching system by a novel construction of sub- and supersolutions and application of Perron's method. Error bounds for monotone schemes for the HJB equation are then derived from estimates near the boundary, where the standard regularisation procedure for viscosity solutions is not applicable, and are found to be of the same order as known results for the whole space. We deduce error bounds for some common finite difference and truncated semi-Lagrangian schemes.
%
\end{abstract}

\maketitle

\section{Introduction}
This paper derives the regularity of solutions, and error bounds for their numerical approximation, to
the Hamilton-Jacobi-Bellman (HJB) equation
\begin{align}
\label{E} u_t+\sup_{\alpha \in \A} \mathcal{L}^{\alpha}(t, x,u,Du,D^2u)=0& && \text{in } Q_T,\\
\label{IC}  u(0, x) = \Psi_0(x)& &&\text{for } x \in \overline{\Omega},  \\
\label{BC} u(t, x) = \Psi_1(t, x)& &&\text{for } (t,x) \in (0, T] \times \partial \Omega,
\end{align}
where $\Omega$ is an open and bounded subset of $\mathbb{R}^d$, $Q_T \defeq (0, T] \times \Omega$, $\overline{\Omega} \defeq \Omega \cup \partial \Omega \subset \mathbb{R}^d$, $\mathcal{A}$ is a compact metric space, 
 $\mathcal{L}^{\alpha}:(0, T]\times \Omega \times \R \times\R^d \times \R^{d\times d} \to \R$ defined as
\begin{align} \label{eq:def_L}
\mathcal{L}^{\alpha}(t, x, r, q, X) = -\text{tr}[a^{\alpha}(t, x) X] 
	- b^{\alpha}(t,x)q - c^{\alpha}(t, x) r - \ell^{\alpha}(t, x)
\end{align}
is a second order differential operator, $\Psi_0$ and $\Psi_1$ are the initial and boundary data, respectively. 
The coefficients $a^{\alpha}$, $b^\alpha$, $c^\alpha$ and $\ell^\alpha$ take values, respectively, in $\mathcal{S}^d$, the space of $d\times d$ real symmetric matrices, $\R^d$, $\R$ and $\R$, where $a^{\alpha}$ is assumed to be positive definite, not necessarily strictly. 
We denote by $\partial^*Q_T$ the parabolic boundary of $Q_T$, i.e.\ $\partial^*{Q}_T \defeq (\{0\} \times \overline{\Omega}) \cup  ((0, T] \times \partial \Omega)$.
For compactness, we define
\begin{align}
\label{eq:def_F}
F(t, x, r, q, X) \defeq \sup_{\alpha \in \mathcal{A}} \mathcal{L}^{\alpha}(t, x,r,q,X),
\end{align}
where the operator $\mathcal{L}^{\alpha}$ is given in \eqref{eq:def_L}.

The essential difference to previous studies, in particular \cite{barles2007error}, is that we consider the equation on domains with Dirichlet data, rather than on $\mathbb{R}^d$, which opens up challenging questions regarding the regularity of the solution in the vicinity of the boundary, and the construction and analysis of approximation schemes there.
The case of initial-boundary value problems is practically relevant not only when the original problem is posed on a bounded domain, but also
when the original problem is posed on the whole space and a localisation to a bounded domain is required for computational tractability. Usually, asymptotic approximations to the values at this artificial boundary are set in this case.

Here, we assume that the solution to \eqref{E}--\eqref{BC} satisfies Dirichlet boundary data pointwise. 
It is well-known that 
for degenerate equations,
the solution may not satisfy the boundary conditions but instead
the equation may hold up to the boundary.
This roughly means that the paths of the underlying controlled stochastic process departing from $\partial \Omega$ stay within the domain for a small time regardless of the value of the control, see \cite{barles1998strong}. 
For classical results on
linear parabolic or elliptic PDEs with non-negative characteristic forms at the boundaries see \cite{oleinik1973second, freidlin1985functional, friedman1975stochastic_v2}.
Comparison results for HJB equations on smooth spatial domains are given in \cite{barles1998strong}, and subsequently
under weaker smoothness assumptions on the domain in \cite{chaumont2004uniqueness}. 

The assumptions we make on the domain are identical to \cite{KrylovDong07Estimates} and formulated in terms of the existence of a smooth ``barrier function'', essentially a positive strict supersolution which is zero at the boundary.
{Under this assumption, the existence of a continuous solution to \eqref{E}--\eqref{BC} is shown in \cite{KrylovDong07Regular} using a stochastic representation result for the HJB equation.
From the probabilistic point of view, the barrier assumption} can be interpreted as ensuring that the expected exit time of the controlled stochastic process goes to zero as the boundary is approached. 
It subsequently allows estimates of the solution and its numerical approximations near the boundary.

{We first extend these well-posedness and regularity results to the switching system}
\begin{align}
\label{eq:A1}
F_i(t, x,u,\partial_tu_i,Du_i,D^2u_i)&=0 && \text{in }\quad Q_T,\\
\label{eq:ICA1}
u_i(0, x) &= \Psi_{0}(x), && \text{for}\quad x \in \overline{\Omega}, \\
\label{eq:BCA1}
u_i(t, x) &= \Psi_{1}(t,x), && \text{for}\quad (t, x) \in (0, T] \times \partial \Omega,
\end{align}
for all $i\in\II:=\{1,\dots,M\}$, where $u=(u_1,\ldots,u_M)$, and with 
\begin{align}
\label{eq:defFi}
F_i(t,x,r,p_t,p_x,X)&
=\max\Big\{p_t +
\sup_{\alpha\in\A_i}\mathcal{L}^{\alpha}_i(t, x,r_i,p_x,X);\   
r_i-\mathcal{M}_i r\Big\},\\
\label{eq:defLi}
\mathcal{L}^{\alpha}_i(t,x,s,q,X)&=-\mathrm{tr}[a^{\alpha}_i(t,x) X] -
b^{\alpha}_i(t,x) q - c^{\alpha}_i(t,x) s -\ell^{\alpha}_i(t,x), \\
\label{eq:defM}
\mathcal{M}_i r &= \min_{j \neq i} \{ r_j + k \}, 
\end{align}
$r \in \R^M$, and $k>0$ is a constant switching cost.
{We then analyse 
the convergence} of its components to the solution of the HJB equation as $k\rightarrow 0$.
This extends the results in \cite{barles2007error} to domains and is of independent interest.
Crucially, this 
requires new technical estimates near the boundaries.

A probabilistic interpretation can be given for the solution to the switching system \eqref{eq:A1}--\eqref{eq:BCA1} as the value function of a controlled optimal switching problem (see for instance \cite{pham2009continuous} for the infinite horizon case). 
Although it should be possible to {thus} extend the technique used in \cite{KrylovDong07Regular}
{(to derive the existence of a continuous solution to \eqref{E} satisfying pointwise the boundary conditions from its stochastic representation)} to the switching system,
we follow a different path and construct suitable sub- and supersolutions satisfying the initial and boundary data to deduce the existence of a continuous viscosity solution by Perron's method. This gives a much shorter proof and allows us to stay fully in an analytic framework.

We moreover provide a continuous dependence result and a comparison principle for \eqref{eq:A1}--\eqref{eq:BCA1}, where the solution close to the boundary is controlled essentially by the barrier function, and deduce convergence of order 1/3 in the switching cost.

The final aim of the analysis is to estimate the difference between the viscosity solution of the HJB equation and an approximate solution computed by means of a numerical scheme.
Let $\G_h \subset \overline{Q}_T$ be a discrete grid with refinement parameter $h$, then fully discrete numerical schemes for \eqref{E}--\eqref{BC} can be written as
\begin{align}
\label{S}
S(h,t,x,u_h(t,x),[u_h]_{t,x})&=0 &&\text{in}\quad
\mathcal{G}_h^+ \defeq \mathcal{G}_h \setminus\ (\{t=0\} \cup \partial \Omega),\\ 
u_h(0,x)&=\Psi_{h,0}(x) && \text{in}\quad \mathcal{G}_h^0 \defeq \mathcal{G}_h\cap\{t=0\},\nonumber \\
u_h(t,x)&=\Psi_{h,1}(t, x) && \text{in}\quad \mathcal{G}_h^1 \defeq \mathcal{G}_h\cap ((0,T] \cap \partial \Omega) ,\nonumber
\end{align}
where
{
$[u_h]_{t,x}$ denotes the numerical solution for $\mathcal{G}_h \backslash (t,x)$.
We will assume that
}
$S$ is a consistent, monotone and uniformly
continuous approximation of the equation \eqref{E} on the
grid $\mathcal{G}^+_h$ {
in the usual sense, which will be made precise later}.
By analogy to the continuous case we denote $\partial^* \mathcal{G}_h \defeq \mathcal{G}_h^0 \cup \mathcal{G}_h^1$.
Following the notation in \cite{barles2007error}, we say that any function  $u_h:\mathcal{G}_h \to \mathbb{R}$ is a grid function and, if finite, belongs to $C_b(\mathcal{G}_h)$, the space of bounded and continuous grid functions.
As we are interested in discrete $\mathcal{G}_h$, as noted in \cite{barles2007error}, any grid function on $\mathcal{G}_h$ is continuous.

The objective is thus to find upper and lower bounds for the difference $u - u_h$.
A central element of the analysis is the use of Krylov's ``shaking coefficients" method (see \cite{krylov1997rate, krylov2000, KrylovDong07Estimates}) to find perturbed equations from which to construct smooth approximations to $u$ and, under certain regularity of the numerical scheme, to $u_h$.
This allows use of the truncation error to bound $u - u_h$.

Key to the approach is the convexity (or concavity) of \eqref{E}, which is used to prove that mollified subsolutions (supersolutions) of \eqref{E} 
are still subsolutions (supersolutions), see Lemma 2.7 in \cite{barles2002convergence}, and gives upper (or lower) bounds.
Without this convexity (or concavity) the error analysis yields weaker results. 
For instance, \cite{caffarelli2008Rate} proves the existence of an algebraic rate of convergence for the finite difference approximation of
$F(D^2 u) = \ell(x),$
on a regular domain with Dirichlet boundary data. 
This result is extended to Isaacs equations in \cite{krylov2015rate}.
However, neither of these articles provide an explicit way to calculate the rates, which may depend on the constant of ellipticity (see \cite{krylov2015rate}).

Hence, while convexity (concavity) allows us to build smooth subsolutions (supersolutions) and upper (lower) bounds, by ``shaking the coefficients" of the equation, 
we cannot directly construct smooth supersolutions (subsolutions) and 
the other bound. 
Two different main approaches have been developed in the literature. 

The first approach, applied in \cite{KrylovDong07Estimates}, treats the equation and the scheme symmetrically and constructs smooth subsolutions to both the HJB equation \eqref{E} and the scheme \eqref{S}.
This procedure requires sufficient regularity of the solution and continuous dependence estimates on the boundary data and the coefficients for both \eqref{E} and the scheme \eqref{S}.
Such results for solutions of equation \eqref{E} have been proved in \cite{KrylovDong07Regular}, under suitable conditions, by means of probabilistic arguments.
For problems posed on spatial domains, Krylov's regularization has also been previously applied in \cite{Picarelli2015dynamic} for the particular case of a semi-infinite domain with an oblique derivative condition. 

The second approach, used in \cite{barles2007error} on $\R^d$,  also derives one of the bounds by ``shaking the coefficients" of the equation {to produce a smooth subsolution}, but for the other bound, 
a key tool for the analysis is the approximation of \eqref{E}--\eqref{BC} by a switching system of the  type \eqref{eq:A1}--\eqref{eq:defM}.
As observed by the authors of \cite{barles2007error}, this procedure 
can be applied to a wider class of schemes, but results in lower rates.
We follow this second  approach 
for the lower bound, in order to deal with the unavoidable complexity of general monotone schemes. 
For the application of this approach to derive error bounds for semi-Lagrangian schemes when $\Omega = \R^d$ see \cite{debrabant2013semi}. The fundamental difference of our analysis to that in \cite{barles2007error} and \cite{debrabant2013semi} is that we consider a bounded domain $\Omega$ with Dirichlet conditions, with the extra technical difficulties this entails.

The range of applicability of the present analysis crucially extends the one in \cite{KrylovDong07Estimates},
which considers a specific semi-discrete scheme  -- essentially, a semi-Lagrangian scheme without interpolation -- which is practically not feasible as the solution is not fully defined
on a fixed mesh. Rather, the solution at a fixed point has to be constructed by a multinomial tree whose nodes depend on the controls, where their number grows exponentially with the number of timesteps.
In contrast, the analysis here is applicable to a class of current state-of-the-art fully discrete monotone schemes, including the different variants of (non-local) semi-Lagrangian \cite{camilli1995approximation, debrabant2013semi, falcone2013book} or hybrid schemes \cite{ma2016unconditionally}, and, under conditions on the diffusion matrix, the (local) seven-point stencil (see, e.g., \cite[Section 5.1.4]{hackbush1992elliptic} or \cite[Section 5.3.1]{kushner2001numerical}).
These more complicated schemes need a different approach to the analysis compared to the one proposed in \cite{KrylovDong07Estimates}, especially for the lower bound,
as explained above. 

As fully discrete monotone schemes generally require a ``wide stencil'' (i.e., involving not only a fixed number of neighbouring nodes), a modification is needed near the boundary.
In \cite{KrylovDong07Estimates}, it is assumed that smooth ``boundary'' data are defined on the whole space, including outside the domain.
Instead, we truncate the wide stencil schemes close to the boundary and modify their coefficients to ensure consistency, albeit usually at a reduced order \cite{ReisingerRotaetxe17}. 
This requires Dirichlet data only on the boundary for the definition of the scheme. 
Interestingly, however, as the solution near the boundary is controlled through the barrier function, the (reduced) order of the truncated scheme at the boundary does not enter in the analysis
and we obtain the same global convergence order as in \cite{barles2007error} on the whole space.
This coincides with the empirical evidence in \cite{ReisingerRotaetxe17, picarelli2017boundary} that the presence of the boundary does not affect the convergence order.

The rest of the paper is organised as follows. 
Section \ref{sec:resultVisc} compiles definitions and results for viscosity solution used throughout the paper.
Section \ref{sec:Fresults} derives some fundamental theoretical results on switching systems with Cauchy-Dirichlet boundary conditions  and estimates the convergence rate of the switching system to a related HJB equation.
Section \ref{sec:Errors} provides the main error bounds for a generic monotone and stable finite difference scheme in terms of the truncation error for a regularised solution.
Section \ref{sec:ErrExamples} then deduces concrete error bounds for two examples of finite difference and semi-Lagrangian schemes from the literature.
Section \ref{sec:conc_err} concludes and suggests directions for further research.

\section{Definitions and general results for HJB equations in 
domains}
\label{sec:resultVisc}

This section contains definitions and background results for HJB equations used throughout the rest of the paper.

\label{sec:defVisc}

We recall that, for a domain $Q_T$, we denote by $\overline{Q}_T$ its closure and by $\partial^*{Q}_T$ the parabolic boundary, i.e.\ $\partial^*{Q}_T \defeq (\{0\} \times \overline{\Omega}) \cup  ((0, T] \times \partial \Omega)$. 
We denote by $\leq$ the component by component ordering in $\mathbb{R}^d$ and the ordering in the sense of positive semi-definite matrices in $\mathcal{S}^d$.

Let $\phi : Q \to \mathbb{R}^d$ be a bounded function from some set $Q$ into $\R^d$ with $d\geq 1$, then the following function norms are used 
$$|\phi|_0 \defeq \sup_{(t, y) \in Q} |\phi(t, y)|,$$
and for any $\delta \in (0, 1]$,
$$[\phi]_\delta \defeq \sup_{(t, x) \neq (s, y)} \frac{|\phi(t, x) - \phi(s, y)|}{(|x-y|+|t-s|^{1/2})^\delta}, \quad \text{and} \quad |\phi|_\delta \defeq |\phi|_0 + [\phi]_\delta.$$ 
As usual, we denote by $C^{n,m}(Q)$ the space of continuous functions $n$-times differentiable in $t$ and $m$ in $x$. If $n=m$ we will simply write $C^{n}(Q)$ and $n = 0$ is used for the space of bounded continuous functions in $Q$. Additionally, for $\delta \in (0,1]$,
${C}^{0}_{\delta}$ denotes the subset of $C^0$ with finite $|\cdot|_\delta$ norm and 
$C^{2}_{\delta}(Q)$ the subset of $C^{1,2}(Q)$ of functions with finite norm
$$
|\phi|_{2,\delta}:= \sum_{\substack{(\beta_0,\beta)\in\mathbb N_0\times\mathbb N_0^d\\2\beta_0+|\beta|\leq 2}} |\partial_t^{\beta_0}D^{\beta} \phi |_\delta.
$$ 

For the regularisation we will take convolutions of functions with the following family of mollifiers in time and space
\begin{align}
\label{eq:rhoxt}
\rho_\eps(t, x) \defeq 
	\frac{1}{\eps^{d+2}}\rho \left( \frac{t}{\eps^2}, \frac{x}{\eps} \right),
\end{align}
where $\eps > 0$, and
\[
\rho \in C^\infty(\R^{d+1}),\quad \rho\geq 0,\quad 
\mathrm{supp}\,\rho = (0,1) \times \{|x| < 1\}, \quad
\int_{\mathrm{supp}\,\rho}\rho(e) \diff e =1.
\]

{ A family of mollifiers in space only is defined similarly and we do not distinguish them notationally for simplicity.}


Let $Q$ be an open set and $d \in \mathbb{N}$, for a locally bounded function $\phi : Q \to \mathbb{R}^d$ we define its upper-semicontinuous envelope
$$\phi^*(x) = \limsup_{\substack{y \to x\\ y \in Q}} \phi(y),$$
and its lower-semicontinuous envelope
$$\phi_*(x) = \liminf_{\substack{y \to x\\ y \in Q}} \phi(y).$$
We denote by $\mathrm{USC}(Q;\R^d)$ and $\mathrm{LSC}(Q;\R^d)$ the usual spaces of upper- and lower-semicontinuous functions $Q \to \mathbb{R}^d$, respectively. 

The relevant notion of solutions for the type of non-linear equations (\ref{E})--(\ref{BC}) is that of viscosity solutions (see \cite{crandall1992user} for a detailed overview). 
{
In the next definition we recall the notion of solution 
when the boundary and initial conditions are satisfied in the ``strong sense".
\begin{definition}[\,{\bf Viscosity solution}\,] \label{def:sol}
A function $\overline{u} \in \mathrm{USC}([0, T] \times \overline{\Omega}; \R)$ 
is a viscosity subsolution, if for each function $\varphi \in C^{1,2}([0, T] \times \overline{\Omega})$, at each maximum point $(t, x)$ of $\overline{u} - \varphi$ we have that
\begin{align}
\nonumber
\varphi_t + F(t, x, \overline{u}, D \varphi, D^2\varphi) &\leq 0, 
& (t,x) \in (0, T] \times \Omega, \\
\nonumber
\varphi - \Psi_0
&\leq 0, & (t,x) \in \{0\} \times \overline{\Omega}, \\
\varphi - \Psi_1
&\leq 0, & (t,x) \in (0, T] \times \partial \Omega.
\label{strong_bc}
\end{align}
Similarly, a function $\underline{u} \in \mathrm{LSC}([0, T] \times \overline{\Omega}; \R)$ 
is a viscosity supersolution, if for each function $\varphi \in C^{1,2}([0, T] \times \overline{\Omega})$, at each minimum point $(t, x)$ of $\underline{u} - \varphi$ we have that
\begin{align*}
\varphi_t + F(t, x, \underline{u}, D \varphi, D^2\varphi) &\geq 0, 
& (t,x) \in (0, T] \times \Omega, \\
\varphi - \Psi_0
&\geq 0, & (t,x) \in \{0\} \times \overline{\Omega}, \\
\varphi - \Psi_1
&\geq 0, & (t,x) \in (0, T] \times \partial \Omega.
\end{align*}
Finally, a continuous function $u$ is a viscosity solution of \eqref{E}--\eqref{BC} if it is both a subsolution and a supersolution.
\end{definition}

{
This definition of viscosity solutions is formulated in terms of smooth test functions $\varphi$. 
It is straightforward to rephrase it in terms of parabolic semijets, see Definition \ref{def:semijets}.\footnote{The use of semijets permits the representation of ``$(Du, D^2 u)$'' for non-differentiable functions $u$. This turns out to be useful in the formulation of the  Crandall-Ishii lemma, see Theorem 8.3 in \cite{crandall1992user} or Theorem \ref{thm:maxprinc} in the appendix, which is the main tool to obtain a maximum principle for semi-continuous functions.}}


Under some structural assumptions on the operator $F$, uniqueness for continuous solutions (in the sense of Definition \ref{def:sol}) of \eqref{E}--\eqref{BC} is proved in \cite[Theorem V.8.1 and Remark 8.1]{fleming2006controlled} as a corollary of a comparison result. 
\label{sec:Regular}

We end this section by a brief illustration of the effect of Dirichlet conditions for degenerate equations 
on numerical schemes. 
In the following example,
the 
uniqueness of
\emph{discontinuous} viscosity solutions fails for boundary conditions in the \emph{weak} sense of Definition 1.1 in \cite{barles1991convergence},
where  (\ref{strong_bc}) is replaced by
\begin{align*}
\min(\varphi_t + F(t, x, \overline{u}^*, D \varphi, D^2\varphi), \overline{u}^*-\Psi_1) \leq 0, & \quad (t,x) \in (0, T] \times \partial \Omega,
\end{align*}
for the subsolution property at the boundary and similar for the supersolution property
(see also \cite{barles1998strong} for a detailed analysis of comparison principles under degeneracy,
and Proposition 2.1 in \cite{jensen2017notion} for a similar observation to ours for the Monge-Amp{\`e}re equation).
As a consequence,
uniform convergence up to the boundary fails.
Instead, the numerical solution forms a boundary layer where it transits from the unique solution in the interior to
the artificial boundary value (see \cite{jakobsen2010monotone}).
Assumption (A2) in Section \ref{sec:Fresults} rules out such behaviour.

\begin{example}
Consider the following second order parabolic PDE
\begin{align}
\label{eq:ExamplePDE}	
u_t - \frac{1}{2} x^2(1-x)^2 u_{xx} + u &=0, && \text{for } (t,x) \in (0, T]\times (0, 1),\\
\label{eq:ExampleBC}	
u(t, x) &= 1, &&\text{for } (t, x) \in (\{0\} \times [0, 1]) \cup ((0, T] \times \{0, 1\}).
\end{align}
It follows 
 that
any $u$ satisfying
\begin{align*}
u(t,x) =  {e}^{-t}, && \quad 
(t,x) \in (0, T] \times (0, 1) \cup \{0\} \times [0,1],\\
{e}^{-t} \le u(t,0), u(t,1) \le 1, && \quad  t \in (0,1],
\end{align*} 
is a viscosity solution to (\ref{eq:ExamplePDE}), (\ref{eq:ExampleBC}) in the sense of \cite{barles1991convergence}. 
Explicitly,
$u^* \le 1$ at the boundary, so that 
the subsolution property holds,
and $u_* = {e}^{-t}$ at the boundary, 
so that we have the supersolution property as the differential equation is satisfied.
Consequently, the viscosity solution is not unique.\footnote{It is also clear that there is no continuous viscosity solution in the sense of Definition \ref{def:sol}.}

We now show that a stable, consistent and monotone scheme may fail to converge uniformly to \emph{any} of these solutions.
Let $\Delta t, \Delta x \geq 0$, $N \defeq T/ \Delta t$ and $J \defeq 1/\Delta x$, then a possible numerical scheme for the approximation of \eqref{eq:ExamplePDE}--\eqref{eq:ExampleBC}	is the following explicit finite difference scheme
\begin{align*}
S(h,t_n,x_j,U_j^n,[U]_{n, j})= 	
\frac{U^n_j - U^{n-1}_j\!\!\!}{\Delta t} \!\!\; - \!\!\; \frac{1}{2} j^2(1\!\!\;-\!\!\;x_j)^2 (U^{n-1}_{j+1} \!\!\;-\!\!\; 2U^{n-1}_j \!\!\!\;+\!\!\; U^{n-1}_{j-1}) \!\!\;+\!\!\; U^{n-1}_j\!\!\!\!\;\;,
\end{align*}
where $h = (\Delta t, \Delta x)$, $n \in [1, N]$, $j \in [1, J-1]$, $t_n = n \Delta t$, $x_j = j \Delta x$, and $U_j^n \equiv U(t_n, x_j)$.
The scheme enforces the initial and boundary conditions. 
One can easily verify that the scheme is monotone and $L^\infty$-stable provided that $\Delta t \le 16 \Delta x^2$.

From the limits below, the numerical solution at the node with $j = 1$ is seen to converge to a constant different from 0 as
$t\rightarrow \infty$,
\begin{align}
U^{n+1}_1 &= \frac{\Delta t}{2}(1-\Delta x)^2[U^n_2 + 1] + (1 - \Delta t - \Delta t(1-\Delta x)^2)U^n_1 \nonumber \\
\label{ineq:Example} &\geq \frac{\Delta t}{2} { (1-\Delta x)^2} + (1 - 2\Delta t) U^n_1 \\
\label{eq:LimitExample} &\geq \sum^n_{m=0}(1 - 2\Delta t)^m \frac{\Delta t}{2}  { (1-\Delta x)^2} + (1 - 2\Delta t)^{n+1}  \rightarrow 
\frac{1 + 3 e^{-2 t}}{4} > \frac{1}{4}
\end{align}
for $n\rightarrow \infty$ with $n \Delta t = t$ fixed,
where 
\eqref{ineq:Example} is obtained from 
$U^n_1, U^n_2 \geq 0$.

Therefore, 
we deduce that for $t > \ln(4)$ the scheme cannot converge uniformly.
\end{example}

\section{Switching systems with Dirichlet boundary conditions}
\label{sec:Fresults}


In this section we will study the following switching system \eqref{eq:A1}--\eqref{eq:BCA1}.
{
The definition of viscosity solutions, sub- and supersolutions for (\ref{eq:A1})--(\ref{eq:BCA1}) is an obvious extension from Definition \ref{def:sol}.
}

Let us consider the following assumptions on the coefficients and boundary data, which are
very similar to assumption 2.2 in \cite{KrylovDong07Estimates}; see also \cite{KrylovDong07Regular} for the introduction of
the barrier function.
Note that $\Psi_1: \overline{Q}_T \rightarrow \mathbb{R}$, i.e., it is required on all of $\overline \Omega$.

\smallskip
\noindent {\bf (A1)} {\bf (Regularity of the coefficients)}  For any $i\in \II$, $\A_i$ is a compact 
metric space.
For any $i\in \II$ and $\alpha\in \A_i$, let
$a^{\alpha}_i=\frac{1}{2}\sigma^{\alpha}_i{\sigma^{\alpha, }_i}^T$
for some $d\times P$ matrix $\sigma^{\alpha}_i$. Furthermore,
there is a constant $C_0\geq 0$ 
independent of $i,\alpha$, such that
$$
[\Psi_{0}]_1+ 
|\sigma^{\alpha}_i|_1+|b^{\alpha}_i|_1 
+|c^{\alpha}_i|_1+|\ell^{\alpha}_i|_1\leq C_0.$$
\medskip
\noindent {\bf (A2)} {\bf (Barrier function)}
There exists a function $\zeta\in C^{1,2}(\overline Q_T)$, such that 
\begin{eqnarray}
\zeta > 0 && \text{ in } Q_T, \label{eq:zeta0}\\
\zeta = 0 && \text{ in } (0,T]\times \partial \Omega, \label{eq:zeta1}
\end{eqnarray}
and for every $i\in\mathcal I, \alpha\in\mathcal A_i$ 
\begin{equation}
\label{zeta_inequ}
-\zeta_t +b^\alpha_i D\zeta +tr[a^\alpha_i D^2\zeta] + c^\alpha_i \zeta \leq -1\qquad \text{in } Q_T.
\end{equation}

\begin{remark}[{Existence of barrier function}]\label{rem:barrier}
The central assumption (A2) merits a detailed discussion.
{We begin with the non-degenerate case.}

For a strictly elliptic second order operator, i.e.\ if $\xi^T a^\alpha_i\xi \geq \lambda |\xi|^2, \forall \xi\in \R^d$, with $c^\alpha_i \le 0$, and for a sufficiently smooth boundary,
one expects there to be a classical solution to the {stationary} equation
\begin{equation*}
\sup_{i\in\mathcal I, \alpha\in\mathcal A_i} \left\{ b^\alpha_i D\zeta +tr[a^\alpha_i D^2\zeta] + c^\alpha_i \zeta\right\} = -1\qquad \text{in } {\Omega}
\end{equation*}
with zero Dirichlet conditions on the boundary $\partial\Omega$.
{Theorem 3.2 in \cite{safonov} gives detailed conditions for the existence of a solution which is $C^{2+\alpha}$ up to the boundary for some $\alpha>0$.}
This function can serve as barrier function, where the strict positivity in the interior follows from the strong maximum principle.

Some cases of degeneracy, including examples with nonsmooth domains, are discussed in \cite[Example 2.3]{KrylovDong07Estimates}.

{We now give some sufficient conditions in the degenerate case.}
If $\partial \Omega \in W^{3,\infty}$, the distance  $d$ to the boundary $\partial\Omega$ is of class $C^2$ in a neighborhood of  $\partial \Omega$. In this case assumption (A2) is satisfied as soon as there exists $\eta>0$ and a neighborhood $\Omega'$ of $\partial\Omega$ such that for every $i\in\mathcal I, \alpha\in\mathcal A_i$
\begin{equation}\label{eq:cond_dist}
-b^\alpha_i n(x)  + tr[a^\alpha_i D^2d(x)]\leq -\eta \qquad \text{in } (0,T]\times \Omega'
\end{equation} 
(where $n(\cdot)$ denotes the outward normal vector), which is seen by taking $$\zeta(t,x)=\frac{1}{\eta}e^{\lambda (t-T)} \tilde d(x)$$ for $\lambda>0$ big enough and
$\tilde d\geq d$ a $C^2(\overline{\Omega})$ function coinciding with $d$ in $\Omega'$.

Condition \eqref{eq:cond_dist} can be generalised to the case of a nonsmooth domain requiring the existence of a function $\rho\in C^2(\R^d)$ such that 
\begin{equation}
\label{eq:cond_general}
\begin{split}
& \rho = 0,\; D\rho\neq 0 \text{ on } \partial \Omega, \quad \rho>0  \text{ on } \Omega;\\
& \text{there exists $\eta>0$ and a neighborhood $\Omega'$ of $\partial\Omega$ such that }\\
& b^\alpha_i D\rho(x)  + tr[a^\alpha_i D^2\rho(x)]\leq -\eta \qquad \text{in } (0,T]\times \Omega', \forall i\in\mathcal I, \alpha\in\mathcal A_i.
\end{split}
\end{equation} 

According to the analysis in \cite{barles1998strong} (resp.\ \cite{chaumont2004uniqueness}), condition \eqref{eq:cond_dist} (resp.\ \eqref{eq:cond_general}) implies that for any viscosity solution of \eqref{eq:A1}-\eqref{eq:BCA1} the Dirichlet boundary condition is satisfied pointwise (compare \eqref{eq:cond_dist} with the conditions defining the set $\Gamma_{\!\mathrm{out}}$ in \cite{barles1998strong} and \cite{chaumont2004uniqueness}).
In particular, considering the dynamics of the control problem associated with \eqref{eq:A1}-\eqref{eq:BCA1}, conditions  \eqref{eq:cond_dist} and \eqref{eq:cond_general} imply that  in a neighborhood of the boundary all trajectories are pushed outside $\Omega$.

However, condition (A2) precludes {some simple cases where the non-smoothness of the boundary is incompatible with regular solutions to the PDE,
such as the heat equation on the unit square. Here,} the zero boundary conditions dictate that all derivatives vanish in the corners,
which implies there cannot be a function $\zeta \in C^2(\overline{\Omega})$ which satisfies \eqref{zeta_inequ}.
\end{remark}


\medskip
\noindent {\bf (A3)} {\bf (Boundary condition)}
$\Psi_1\in C^{1,2}(\overline Q_T)$. Moreover, there exists a constant $C_1>0$ such that  
\begin{eqnarray}\label{eq:psi0psi1}
|\Psi_0-\Psi_1(0,\cdot)| \leq C_1 \zeta(0,\cdot)  && \text{ on } 
\overline \Omega.
\end{eqnarray}

\begin{remark}\label{rem:boundA}
Assumption (A3) is a compatibility condition between the initial and boundary data.
The regularity assumption on $\Psi_1$ and $\zeta$ and the boundedness of the domain $\overline Q_T$ imply the boundedness of $\Psi_1$ and $\zeta$ together with their derivatives. This fact will be used strongly for the construction of smooth sub- and supersolutions.
\end{remark}
}

\subsection{Preliminary results}\label{sec:switch_prelims}

We seek to establish  existence, uniqueness, and regularity results for system \eqref{eq:A1}--\eqref{eq:BCA1}. Equation \eqref{E}--\eqref{BC} represents a particular case of \eqref{eq:A1}--\eqref{eq:BCA1}
when $|\mathcal I|=1$, hence this section also contains all the elements necessary for the study of \eqref{E}--\eqref{BC}.

\begin{theorem}[\,{\bf Maximum principle}\textsf{\,}]
\label{thm:WP}
Let assumption (A1) be satisfied. 
If $u\in USC(\overline Q_T;\R^M)$ is a subsolution of \eqref{eq:A1} 
and 
$v\in LSC(\overline Q_T;\R^M)$ a supersolution, 
then $u-v\leq \underset{\partial^*Q_T}\sup (u-v)$ in $\overline Q_T$.
\end{theorem}
\begin{proof}  
We adapt the proof of Theorem 8.2 in \cite{crandall1992user} using the parabolic version of 
the Crandall-Ishii lemma \cite{crandallIshii}, {
see Theorem \ref{thm:maxprinc} in the appendix for convenience}. 

We assume by contradiction that 
$$u_i(s, z)- v_i(s, z) > \underset{\II\times\partial^*Q_T}\sup (u_j-v_j)$$
for some $(i, s, z) \in \II \times \overline Q_T$. 
We start by noticing that for any  $\rho > 0$, $u^\rho = u - \rho/(T-t)$ is a subsolution of \eqref{eq:A1}. For $\rho$ small enough, we can define $(\bar i, \bar t, \bar x)\in \II\times Q_T$ such that
$$
u^\rho_{\bar i}(\bar t,\bar x) -v_{\bar i}(\bar t, \bar x)= \underset{\II\times\overline Q_T}\sup (u^\rho_i-v_i)>\underset{\II\times\partial^*Q_T}\sup (u^\rho_i-v_i).
$$  
For some $\beta > 0$ consider the auxiliary function
\begin{align*}
\Phi (i, t, x, y) = u^\rho_i(t, x) - v_i(t, y) - \beta |x-y|^2 , \qquad t \in [0, T], x, y \in \overline{\Omega}.
\end{align*}
Let $(\hat t,\hat x, \hat y)$\footnote{We omit the dependence of $(\hat{t}, \hat{x}, \hat{y})$ on $\bar i$ and $\beta$ for brevity; explicitly we have $\left(\hat{t}^\beta_{\bar i}, \hat{x}^\beta_{\bar i}, \hat{y}^\beta_{\bar i}\right)$.} 
be a maximum point for $\Phi(\bar i,\cdot,\cdot,\cdot)$. By standard arguments in viscosity theory (see Lemma 3.1 in \cite{crandall1992user}), for $\beta$ big enough we have that  $(\hat{t}, \hat{x}, \hat{y}) \in (0, T) \times \Omega \times \Omega$ and $\beta |\hat{x}-\hat{y}|^2 \to 0$. 

Moreover, by Lemma A.2 in \cite{barles2007error} (see Lemma \ref{lem:Mishii} in the appendix),  there exists $\hat{i} \in \II$ such that $(\hat{i}, \bar{t}, \bar{x})$ is still a maximum point for $u^\rho-v$ and, in addition, $v_{\hat{i}}(\bar{t}, \bar{y}) < \mathcal{M}_{\hat{i}} v(\bar{t}, \bar{y})$. Then, for $\beta$ big enough we can also say that
$v_{\hat{i}}(\hat{t}, \hat{y}) < \mathcal{M}_{\hat{i}} v(\hat{t}, \hat{y})$. 
Now we can make use of the Crandall-Ishii lemma,
Theorem \ref{thm:maxprinc},
with $u^\rho_{\hat{i}}$, $v_{\hat{i}}$ and
$$\phi(t, x, y) = \beta |x-y|^2,$$
to infer that there are numbers $a$, $b$ and symmetric matrices $X, Y \in \mathcal{S}^d$ such that 
\begin{align*}
(a, \beta (\hat{x} - \hat{y}), X) \in \overline{\mathcal{P}}^{2, +}u^\rho_{\hat i}(\hat t, \hat x), \qquad 
(b, \beta (\hat{x} - \hat{y}), Y) \in \overline{\mathcal{P}}^{2, -}v_{\hat i}(\hat t, \hat y),
\end{align*}
satisfying 
\begin{align*}
a-b = 0, \quad \text{and}\quad -3 \beta \begin{pmatrix}
I & 0 \\
0 & I
\end{pmatrix}  \leq \begin{pmatrix}
X & 0 \\
0 & -Y
\end{pmatrix} \leq 3 \beta \begin{pmatrix}
I & -I \\
-I & I
\end{pmatrix}.
\end{align*}

By the sub- and supersolution property of $u^\rho$ and $v$ we have that
\begin{align*}
a + \sup_{\alpha \in \A_{\hat i}} \mathcal{L}^\alpha_{\hat i} (\hat{t}, \hat{x}, u^\rho(\hat{t}, \hat{x}), \beta (\hat{x} - \hat{y}), X) &\leq -\frac{\rho}{T^2}, \\
b + \sup_{\alpha \in \A_{\hat i}} \mathcal{L}^\alpha_{\hat i} (\hat{t}, \hat{y}, v(\hat{t}, \hat{y}), \beta (\hat{x} - \hat{y}), Y) &\geq 0.
\end{align*}
Subtracting these two inequalities and using Lemma V.7.1 in \cite{fleming2006controlled} (stated as Lemma \ref{lm:defModLi} in the appendix) 
we have that
\begin{align*}
\frac{\rho}{T^2} &\leq \sup_{\alpha \in \A_{\hat i}} \mathcal{L}^\alpha_{\hat i} (\hat{t}, \hat{y}, v(\hat{t}, \hat{y}), \beta (\hat{x} - \hat{y}), Y) - \sup_{\alpha \in \A_{\hat i}} \mathcal{L}^\alpha_{\hat i} (\hat{t}, \hat{x}, u^\rho(\hat{t}, \hat{x}), \beta (\hat{x} - \hat{y}), X) \\
&\leq \omega(\beta |\hat{x}-\hat{y}|^2 + |\hat{x}-\hat{y}|) \to 0,
\end{align*}
which leads to a contradiction for $\beta \to \infty$, as $\rho>0$, and concludes the proof. 
\end{proof}


\smallskip
Before we prove the main existence and uniqueness result, we show how the initial datum can be replaced by a smoothed version satisfying the same assumptions as the original.

\begin{lemma}[\,{\bf  Smoothing of initial data}\,]\label{lem:smoothing}
Let $\Psi_0 \in C^0_1(\overline\Omega)$ such that $|\Psi_0|_1 \le C_0$ and $\zeta, \Psi_1 \in C^{1,2}(\overline Q_T)$ such that (\ref{eq:zeta0}), (\ref{eq:zeta1}), and (\ref{eq:psi0psi1}) are satisfied for some $C_1>0$. Then for any $0<\e<1$ small enough there exists $\Psi_{\e} \in C^2(\overline\Omega)$ which satisfies the same conditions as $\Psi_0$, with $C_1$ replaced by some $\overline C_1$ independent of $\e$, and where
\begin{eqnarray}
\label{eq:psie}
|\Psi_0-\Psi_\e| \le C \e
\end{eqnarray}
for some $C>0$ independent of $\e$.
\end{lemma}

\begin{proof}
Without loss of generality consider $\Psi_1(0,\cdot)=0$, 
else we replace $\Psi_0$ by $\Psi_0-\Psi_1(0,\cdot)$ and because $\Psi_1(0,\cdot) \in C^2(\overline \Omega)$ the result follows.
Also assume $C_1=1$ by appropriate scaling of $\zeta$.

Set
\begin{eqnarray}
\label{psihat}
\widehat \Psi_0 &:=& \big((\Psi_0)^+-2 |\zeta|_1 \e \big)^+ + \big((\Psi_0)^-+2 |\zeta|_1 \e \big)^-,
\end{eqnarray}
extended by 0 outside $\Omega$ and where $(x)^\pm$ is the positive/negative part of $x$, respectively. We will show in the remainder that 
\begin{eqnarray*}
\Psi_\epsilon &:=& \widehat \Psi_0*\rho_\eps
\end{eqnarray*}
has the desired properties.

We define for any $\varepsilon>0$
\begin{equation}\label{eqdef:omegaeps}
\Omega_\e := \{x\in\Omega : \zeta_0 > |\zeta|_1 \e \} \subset \Omega^\e:=\{x\in \Omega : d(x)>\e\},
\end{equation}
where {
$\zeta_0(\cdot):=\zeta(0,\cdot)$} and $d$ denotes the distance function to $\partial\Omega$. 

The following properties of $\widehat \Psi_0$ follow directly:
\begin{eqnarray}
\label{eq:psizero}
\widehat \Psi_0 &=& 0 \qquad \qquad \text{in } \; \Omega \backslash \Omega_{2\e}, \\
\label{eq:psihat}
|\widehat \Psi_0 - \Psi_0| &\le& 2 |\zeta|_1 \e, \\
|\widehat \Psi_0|_1 &\le & C_0.
\label{eq:psilip}
\end{eqnarray}
Here, the first property holds because $|\Psi_0| \le {
\zeta_0} \le 2 |\zeta|_1 \e$ in $\Omega \backslash \Omega_{2\e}$ by definition, the 
second and third because $\widehat \Psi_0$ results from a constant vertical shift of $\Psi_0$ up or down whenever $\Psi_0$ is smaller than $-2 |\zeta|_1 \e$
or larger than $2 |\zeta|_1 \e$, respectively, and zero otherwise.

By standard properties of mollifiers, $|\Psi_\e - \widehat \Psi_0| \le C_0 \e$ 
and therefore (\ref{eq:psie}) follows 
with $C= C_0 +  2 |\zeta|_1$, using (\ref{eq:psihat}).

To show (\ref{eq:psi0psi1}), we note that in $\Omega_\e$ from $|\Psi_0|\le \zeta_0$ follows $|\widehat \Psi_0| \leq \zeta_0$  and $|\Psi_\e| \le \zeta_\e := \zeta_0 * \rho_\e$.
Hence, as $\zeta_0 \ge |\zeta|_1 \e$ in $\Omega_\e$ by definition and $\zeta_\e \le \zeta_0 + \e$ by properties of mollifiers, one has $|\Psi_\e| \le \zeta_0 (1+1/|\zeta|_1)$.
In $\Omega\backslash \Omega_\e$, $\Psi_\e = 0$ by (\ref{eq:psizero}) and hence trivially $|\Psi_\e| \le \zeta_0$.
Therefore, (\ref{eq:psi0psi1}) holds in all of $\Omega$ with $C_1$ replaced by $\overline C_1 := 1+1/|\zeta|_1$.
Finally, from (\ref{eq:psilip}) follows also $[\Psi_\e]_1 \le C_0$, which concludes the proof.
\end{proof}

\smallskip

The following theorem gives existence and uniqueness and provides an important control of the solution to \eqref{eq:A1}--\eqref{eq:BCA1} in a neighbourhood of $\partial\Omega$  for proving the regularity result.

\begin{theorem}[\,{\bf Existence and uniqueness}\,]
\label{thm:exist-unique}
Assume (A1), (A2) and (A3) hold. 
Then, there exists a unique continuous viscosity solution $u=(u_1,\ldots, u_M)$ to \eqref{eq:A1}--\eqref{eq:BCA1}.
Moreover, there exists a constant $K>0$ (independent of $M$) such that 
\begin{equation}\label{eq:barrierEstSW}
|u_i(t,x)-\Psi_1(t,x)|\leq K \zeta(t,x), \quad \forall\, (t,x) \in \overline Q_T, \; i\in \cI,
\end{equation}
where $\zeta$ is the function in assumption (A2). 
\end{theorem}
\begin{proof} 
Uniqueness follows by Theorem \ref{thm:WP}. Existence can be proved by Perron's method 
as shown in the context of elliptic equations in \cite[Theorem 4.1]{ishiiKoikeSwitching}.
In the present setting, we have to construct
a lower semicontinuous function $f$ and an upper semicontinuous function $g$ which are, respectively, a sub- and supersolution of \eqref{eq:A1} and satisfy $f=g=\Psi_0$ in $\{0\}\times \overline \Omega$ and $f=g=\Psi_1$ in $(0,T]\times\partial\Omega$.

 We first assume that $\Psi_0 \in C^2(\overline \Omega)$ and will reduce the case of Lipschitz $\Psi_0$ to this case by a regularisation argument at the end of the proof.

Let us start by constructing the subsolution. Let $\lambda:= \sup_{i,\alpha}|c^{\alpha,+}_i|_0$,
\begin{align*}
f_1(t,x):= \Psi_1(t,x) - K_1 \zeta(t,x)\quad\text{and}\quad f_2(t,x):=e^{\lambda t}\left( \Psi_0(x) - K_2\; t\right).
\end{align*}
Thanks to assumption (A2)--(A3), it is easy to verify that
\begin{align*}
f_1(t,x) = \Psi_1(t,x) \text{ on } (0,T]\times\partial\Omega\quad\text{and}\quad  f_2(0,x)= \Psi_0(x) \text{ on } \overline\Omega,
\end{align*}
and taking $K_1\geq C_1$ and $K_2\geq |\partial_t\Psi_1|_0$ one also has
\begin{align*}
f_1(0,x)\leq \Psi_0(x) \text{ on } \overline\Omega\quad\text{and}\quad  f_2(t,x) \leq \Psi_1(t,x) \text{ on } (0,T]\times\partial\Omega.
\end{align*}
Moreover, $(f_1,\ldots, f_1)$ and $(f_2,\ldots, f_2)$ are both subsolutions to \eqref{eq:A1} in $(0, T)\times \Omega$ for $K_1, K_2$ big enough. Indeed, 
$$
f_1 - \underset{j\neq i}\min\{f_1 +k\} = f_2 - \underset{j\neq i}\min\{f_2 +k\} = -k < 0,
$$
and for any $i\in\II$
\begin{align*}
& \partial_tf_1 + \underset{\alpha\in\mathcal A_i}\sup\mathcal L^\alpha_i(t,x,f_1,Df_1,D^2 f_1) \\
& \leq -K_1\zeta_t + K_1\sup_{\alpha\in\mathcal A_i}\Big\{b^\alpha_i(t,x)D\zeta+tr[a^\alpha_i(t,x)D^2\zeta] +c^\alpha_i(t,x)\zeta\Big\} +C\\
& \leq -K_1 + C\leq 0
\end{align*}
for $K_1\geq C$, where $C$ is a constant depending only on the bounds of $\Psi_1$ and its derivatives (see Remark \ref{rem:boundA}) and on the constant $C_0$ in assumption (A1).

To prove that $f_2$ is a viscosity subsolution, 
consider
\begin{align*}
& \partial_t f_2 + \underset{\alpha\in\mathcal A_i}\sup\mathcal L^\alpha_i(t,x, f_2,D f_2,D^2 f_2) \\
&= e^{\lambda t} \left( - K_2 (1+t)  e^{\lambda t} + \Psi_0 + \underset{\alpha\in\mathcal A_i}\sup\mathcal L^\alpha_i(t,x, \Psi_0,D \Psi_0,D^2 \Psi_0)  \right) \le 0
\end{align*}
for $K_2$ large enough which only depends on the constant $C_0$ in assumption (A1) and the derivatives up to order 2 of $\Psi_0$.

At this point, defining for any $(t,x)\in\overline Q_T$
$$
f(t,x):=\max\{f_1(t,x), f_2(t,x)\}
$$
one has 
\begin{equation}\label{eq:f_boundary}
f(t,x) = \Psi_1(t,x) \text{ on } (0,T]\times\partial\Omega\quad\text{and}\quad f(0,x) = \Psi_0(x) \text{ on } \overline\Omega. 
\end{equation}
Recalling that the maximum of viscosity subsolutions is  still a viscosity subsolution (see, e.g., \cite[page 26]{crandall1992user}),
we can conclude that $f$ is the desired function.\\
Analogously one can prove that, defining $g_1(t,x):=\Psi_1(t,x)+K_1\zeta(t,x)$ and $g_2(t,x):=e^{\lambda t}\left(\Psi_0(x)+K_2\; t\right)$,
the continuous function 
$$
g(t,x):=\left(\min\{g_1(t,x),g_2(t,x)\},\ldots,\min\{g_1(t,x),g_2(t,x)\}\right)
$$
is a viscosity supersolution to \eqref{eq:A1}--\eqref{eq:BCA1} satisfying the desired properties on $\partial^*Q_T$.
\\
It remains to prove \eqref{eq:barrierEstSW} which follows easily by Theorem \ref{thm:WP} taking $f_1$ and $g_1$ respectively as sub- and supersolution and $K=\max\{K_1,K_2\}$.

This concludes the proof of the theorem for $\Psi_0 \in C^2(\overline \Omega)$. We now deduce the result for Lipschitz $\Psi_0$.
 {
Let us consider  a sequence of solutions $u_n$ with smooth initial data $\Psi_{1/n}$ as provided by Lemma \ref{lem:smoothing}, and define 
$$
\!f(t,x)\!:=\!\underset{n\to\infty}{\lim\sup}^* u_n(t,x)\!:=\!\lim_{j\to\infty}\sup\Big\{u_n(s,y)\!: n\!>\!j, (s,y)\!\in\!\overline Q_T, |s-t|+|x-y|\!\leq\! \frac{1}{n}\Big\}.
$$
By standard stability arguments for viscosity solution (see \cite[Section 6]{crandall1992user}), $f$ is a viscosity subsolution to \eqref{eq:A1}. It remains to prove that $f$ satisfies the initial and boundary conditions. 
Under assumptions (A1), (A2) and (A3) and for $n$ big enough, the solutions $u_n$ are Lipschitz continuous in $x$ and H\"older-1/2 continuous in time with constants independent of $n$ (see Theorem \ref{thm:WPmain} below). Therefore, 
\begin{align*}
|f(0,x)-\Psi_0(x)|& \leq \lim_{j\to\infty} \underset{\substack{n>j, (s,y)\in \overline Q_T\\ |s|+|x-y|\leq \frac{1}{n}}} \sup\Big\{|u_n(s,y)-u_n(0,x)|+|\Psi_{1/n}(x)-\Psi_0(x)|\Big\}\\
& \leq \lim_{j\to\infty} \underset{\substack{n>j, (s,y)\in \overline Q_T\\ |s-t|+|x-y|\leq \frac{1}{n}}} \sup\Big\{C\Big(|s|^{1/2}+|x-y|+\frac{1}{n}\Big)\Big\} =0. 
\end{align*}
An analogous result holds for the boundary data $\Psi_1$. In conclusion, we have proved that $f$ is a viscosity  subsolution satisfying \eqref{eq:f_boundary}.  
Similarly,
$$g:=
\underset{n\to\infty}{\lim\inf}^* u_n
$$
can be shown to be a viscosity supersolution. 
}

\end{proof}

{
\begin{remark}
Theorem \ref{thm:exist-unique} extends previous work to the Cauchy-Dirichlet problem,
with added time evolution and initial data compared to \cite{ishiiKoikeSwitching} and added boundary data compared to \cite{barles2007error}.
On the other hand, we extend \cite{KrylovDong07Regular} from the HJB case to a switching system. Although we do not foresee any fundamental obstacles in extending the method
in \cite{KrylovDong07Regular} by stochastic representation results to our setting, we give a simpler, purely analytic proof by Perron's method.
\end{remark}
}

Using this result, we can now prove some important regularity properties for viscosity solutions of \eqref{eq:A1}--\eqref{eq:BCA1}. 
Some of the arguments are a straightforward adaptation of Theorem A.1 in \cite{barles2007error}, however, a separate and careful treatment at the boundary is necessary.

\begin{theorem}[\,{\bf Regularity of solutions}\,]
\label{thm:WPmain}
Assume (A1), (A2) and (A3) hold. \\
Let  $u$ be the solution to \eqref{eq:A1}--\eqref{eq:BCA1}, then   $u\in C^{0}_1(\overline{Q}_T)$, i.e.\ the space of bounded continuous functions with finite $|\cdot|_1$ norm,
and satisfies for all $t,s\in[0,T]$
\begin{align*}
e^{-\lambda t} \max_{i\in \cI} |u_i(t,\cdot)|_0\leq \ & 
\max_{j \in \{0, 1\}} |\Psi_{j}|_0 + t \sup_{i,\alpha} |\ell^{\alpha}_i|_0,
\end{align*}
where $\lambda \defeq \sup_{i,\alpha}|c^{\alpha +}_i|_0$,
\begin{align*}
e^{- \lambda_0 t} \max_{i\in \cI}\,[u_i(t,\cdot)]_1\leq \ & 
   [\Psi_{0}]_1+  |D\Psi_1|_0 +K|D\zeta|_0 + t  \sup_{i,\alpha,s} 
	\Big\{|u_i|_0[c^{\alpha}_i(s,\cdot)]_1+[\ell^{\alpha}_i(s,\cdot)]_1\Big\},
\end{align*}
where $\lambda_0 \defeq \sup_{i,\alpha,s} \{|c^{\alpha +}_i(s,\cdot)|_0
+[\sigma^{\alpha}_i(s,\cdot)]_1^2+[b^{\alpha}_i(s,\cdot)]_1\}$, and 
\begin{align*}
\max_{i\in \cI} |u_i(t,x)-u_i(s,x)| \leq \ &C|t-s|^{1/2},
\end{align*}
where
$C$ only depends on $T$, $C_0$
and $\bar{M} \defeq \sup_{i,t}|u_i(t,\cdot)|_1+|\partial_t\Psi_1|_0+2K|D\zeta|_0$.\\

\end{theorem}
\begin{proof}
{Let us start with the boundedness of the solution in the $L^\infty$-norm. 
Setting
$$w(t) \defeq e^{\lambda t}\left(
	\max_{j \in \{0, 1\}} |\Psi_{j}|_0 + t \sup_{i,\alpha} |\ell^{\alpha}_i|_0 		\right),
$$
it is straightforward to verify by insertion (see \cite{barles2007error}) that $w$ is a classical supersolution to \eqref{eq:A1}--\eqref{eq:BCA1}.
Hence by the comparison principle $u_i(t, x) \leq w(t)$ for all $(i, t,x) \in \II \times [0, T] \times \overline{\Omega}$. Proceeding similarly with $-w$ we obtain the bound on $|u|_0$.

}

To establish the Lipschitz regularity of the solution $u$  we start by observing that $u$ is Lipschitz continuous on $[0,T]\times\partial(\Omega\times\Omega)$. This is trivial if
$(x,y) \in\partial\Omega \times \partial\Omega$ or $t=0$. Let now $t>0$, $x\in \partial\Omega$ and $y\in\Omega$. Thanks to \eqref{eq:barrierEstSW}, one has $\forall i\in\mathcal I$
\begin{align}
|u_i(t,x)-u_i(t,y)|& =|\Psi_1(t,x)-u_i(t,y)|\leq |\Psi_1(t,x)-\Psi_1(t,y)| +K\zeta(t,y)\nonumber\\ 
& \leq (|D\Psi_1|_0 +K|D\zeta|_0)|x-y|.\label{eq:lip_bound}
\end{align}
We define
$$m:=\sup_{i,t,x,y}\left\{u_i(t,x)-u_i(t,y)-\bar
w(t)|x-y|\right\},$$
where
\begin{align*} 
\bar w(t):=\ & e^{\lambda_0 t}\Big\{
	 [\Psi_{j}]_1 +|D\Psi_1|_0 +K|D\zeta|_0 + t \sup_{i,\alpha,s} 
\big\{ |u_i|_0[c^{\alpha}_i(s,\cdot)]_1 +
  [\ell^{\alpha}_i(s,\cdot)]_1 \big\}
  \Big\},
\end{align*}

We will follow a similar argument to \cite{barles2007error}, but accounting for the boundaries to prove that $m \leq 0$, from which the result follows. 
We proceed by contradiction assuming that $m>0$ and that the maximum is attained for $\bar i, \bar t, \bar x, \bar y$. 
First of all, we notice if $m>0$ then there exists a $\eta > 0$ such that
$$u_{\bar i}(\bar t,\bar x)-u_{\bar i}(\bar t,\bar y)-\bar w(\bar t)|\bar x-\bar y| - \bar t e^{\lambda_0
  \bar t}\eta>0.$$
Thus, we define an auxiliary function $\psi$ by
$\psi_i(t,x,y) \defeq u_i(t,x)-u_i(t,y)-\bar w(t)|x-y| - t e^{\lambda_0 t}\eta$, 
which also attains a maximum $\bar{M}$ at some point $(\tilde i, \tilde t, \tilde x,\tilde y)$. 
By construction of $\psi_i$ (see the choice of $\eta$), 
the maximum is also strictly positive, i.e.\ $\bar{M}>0$. 
By definition of $\bar w(t)$ we infer that $\tilde t>0$, $\tilde x\neq\tilde y$ and $(\tilde x, \tilde y) \not\in \partial (\Omega \times \Omega)$, by \eqref{eq:lip_bound}. 

Now we check whether the maximum's location can be in the interior of the domain, that is if $(\tilde i, \tilde t, \tilde x,\tilde y) \in \cI \times (0, T) \times \Omega \times \Omega$. 
As noted in \cite{barles2007error}, $\bar w(t)|x-y|+te^{\lambda_0 t}\eta$ is a smooth function at $(\tilde t, \tilde x,\tilde y)$, therefore we can use 
the Crandall-Ishii lemma, 
Theorem \ref{thm:maxprinc},
and Lemma \ref{lem:Mishii} to ignore the switching part for the supersolution to obtain that $\eta\leq 0$.
This is a contradiction and hence $m\leq 0$.

{
Regarding the time regularity, let $t>s$, $\eps>0$ and let $\Omega^\e$ as in \eqref{eqdef:omegaeps}.
We start observing that, thanks to \eqref{eq:barrierEstSW}, in $\overline{\Omega}\setminus\Omega^\e$ the following estimates hold $\forall i\in\mathcal I$
\begin{eqnarray}\nonumber
|u_i(t,x)-u_i(s,x)| &\leq& K\zeta(t,x) + K\zeta(s,x) +  |\partial_t\Psi_1|_0(t-s) \\
&\leq& 2K|D \zeta|_0\e + |\partial_t\Psi_1|_0(t-s).
\label{eq:lipeps}
\end{eqnarray} 

We define $u^\eps_0:= u(s,\cdot)*\rho_\eps$ in $\overline\Omega^\eps$,
where $\rho_\eps$ are mollifiers, 
and consider  the smooth functions
$$w^\pm_{\e, i} (t, x) \defeq 
	e^{\lambda(t-s)}\left(u^\e_{0, i}(x) \pm  C_\e(t-s)\right)\quad \text{and}\quad w^\pm_\eps\equiv(w^\pm_{\e, 1},\ldots,w^\pm_{\e, M}).
$$
It is straightforward to show that $w^+_\e$ (resp. $w^-_\e$) is a supersolution (resp. subsolution) of \eqref{eq:A1} restricted to $(s,T]\times\Omega^\e$, provided that we set
\begin{align*}
C_\eps=C^2_0|D^2 u^\e_0|_0+ C_0(|D u^\e_0|_0+2|u^\e_0|_0+1) \quad \text{and} \quad \lambda = \sup_{i,\alpha}|c^{\alpha, +}_i|_0.
\end{align*}	

We only verify the subsolution property as checking the supersolution property is easier. Thanks to the choice of $C_\eps$, one can easily check that $\forall i\in\mathcal I$
$$
\partial_t w^-_{\e,i} +
\sup_{\alpha\in\A_i} \mathcal{L}^{\alpha}_i(t, x,w^-_{\e,i},e^{\lambda(t-s)}D u^\e_{0,i}, e^{\lambda(t-s)}D^2 u^\e_{0,i})\leq   0.$$
Moreover,  being a subsolution, $u$ satisfies 
$$
u_i(t,x) - \min_{j\neq i}(u_j(t,x)+k)\leq 0,\qquad i\in\II, (t,x)\in Q_T,
$$
which implies 
$$
u^\e_{0,i}(x) - \min_{j\neq i}(u^\e_{0,j}(x)+k)\leq 0, \qquad i\in\II, x\in \Omega^\e.
$$
Hence, $w^-_\e$ is a subsolution in $(s,T]\times\Omega^\e$. Then, applying 
Theorem \ref{thm:WP} one has
(the following inequalities have to be considered componentwise for $i\in\mathcal I$)
\begin{align*}
& u(s,x)-u(t,x) \\
& = u(s,x) - w^-_{\e}(s,x) + w^-_{\e}(s,x) - w^-_\e(t,x) + w^-_\e(t,x) -u(t,x)\\
& \leq \underset{\partial^*Q^\e_{s,T}}\sup(w^-_{\e}-u) + |u^\e_0|_0 (e^{\lambda(t-s)}-1) +e^{\lambda(t-s)}C_\eps (t-s) +u^\e_0(x)-u(s,x)\\
& \leq \underset{\partial^*Q^\e_{s,T}}\sup(w^-_{\e}(t,x)-u(t,x)) + |u|_0\lambda e^{\lambda T}(t-s) +e^{\lambda T}C_\eps (t-s)+[u(s,\cdot)]_1\e
\end{align*}  
 for every $(t,x), (s,x)\in \overline Q^\eps_{s,T}:=[s,T]\times \Omega^\eps$.
 Analogously, using $w^+_\e$, 
\begin{align*}
u(t,x)-u(s,x) \leq \underset{\partial^*Q^\e_{s,T}}\sup(u-w^+_\e) + |u|_0\lambda e^{\lambda T}(t-s) +e^{\lambda T}C_\eps (t-s)+[u(s,\cdot)]_1\e.
\end{align*}  
It remains to estimate $(u-w^+_\e)$ and $(w^-_\e-u)$ on $\partial^*Q^\e_{s,T}=(\{s\}\times \Omega^\e)\cup ((s,T]\times\partial\Omega^\e)$. \\For $t=s$, one has $w^+_\e=w^-_\e=u^\e_0$, hence 
$$
w^-_\e(s,x) - u(s,x) \leq [u(s,\cdot)]_1\e \qquad \text{and} \quad
 \qquad u(s,x)-w^+_\e(s,x) \leq [u(s,\cdot)]_1\e.
$$
For $x\in\partial\Omega^\e$, we use \eqref{eq:lipeps} to majorate $w^-_\e(t,x) - u(t,x)$ and $u(t,x)-w^+_\e(t,x)$ by 
$$
|u|_0\lambda e^{\lambda T}(t-s) +e^{\lambda T}C_\eps (t-s)+[u(s,\cdot)]_1\e+2K|D\zeta|_0\e + |\partial_t\Psi_1|_0(t-s).
$$
At this point a minimization  of the right-hand side  with respect to $\eps$, noting that $C_\eps\leq C_0^2[u]_1\eps^{-1}+C_0([u]_1+2|u|_0+1)$, concludes the proof.
}

\end{proof}

\begin{theorem}[\,{\bf Continuous dependence}\,]
\label{thm:A3}
Let $u$ and $\bar{u}$ be solutions of \eqref{eq:A1}--\eqref{eq:BCA1} with coefficients $\sigma,\,b,\,c,\,\ell$ and $\bar{\sigma},\,\bar{b},\,\bar{c},\,\bar{\ell}$ respectively. 
If both sets of coefficients and the domain satisfy (A1) and (A2), and
$|u|_0+|\bar u|_0+[u(t,\cdot)]_1+[\bar u(t,\cdot)]_1\leq \bar{M} <\infty$ for $t\in
[0,T]$, then 
\begin{align*}
e^{-\lambda t}\max_{i\in \cI} |u_i(t,\cdot)-\bar{u}_i(t,\cdot)|_0 & \leq
\max_{i\in \cI} \sup_{\partial^* Q_T} |u_{i} -\bar u_i|+t^{1/2} \bar K \sup_{i,\alpha}|\sigma^\alpha-\bar{\sigma}^\alpha|_0\\ 
&~~ +t \sup_{i,\alpha}\Big\{2\bar{M}|b^\alpha-\bar{b}^\alpha|_0 + \bar{M} |c^\alpha-\bar{c}^\alpha|_0+|\ell^\alpha-\bar{\ell}^\alpha|_0\Big\},
\end{align*} 
where $\lambda:=\sup_{i,\alpha}|c^-|_0$ and
\begin{align*}
\bar K^2\leq &\ 8\bar{M}^2 +
8\bar{M}T\sup_{i,\alpha}\Big\{ 2\bar{M}[\sigma^\alpha]^2_1\wedge[\bar{\sigma}^\alpha]^2_1\\ 
&+2\bar{M}[b^\alpha]_1\wedge[\bar{b}^\alpha]_1+\bar{M}[c^\alpha]_1\wedge[\bar{c}^\alpha]_1+[\ell^\alpha]_1\wedge[\bar{\ell}^\alpha]_1 \Big\}. 
\end{align*} 
\end{theorem}

\begin{proof}
As done in the proof of Theorem A.3 in \cite{barles2007error}, without loss of generality we assume that $\lambda = 0$.
We start by defining 
\begin{align*}
\psi^i(t,x,y) &\defeq	u_i(t,x)-\bar u_i(t,y)-\frac{1}{\delta}|x-y|^2,\\
m &:=\sup_{i,t,x,y}\psi^i(t,x,y)-\sup_{\II \times Q^*}(\psi^i(t,x,y))^+,\\
\bar m &:=\sup_{i,t,x,y}\left\{\psi^i(t,x,y)-
  \frac{\eta m t}T\right\},
\end{align*}
where $\eta \in(0,1)$ and $Q^* \defeq (\{0\} \times \overline{\Omega} \times \overline{\Omega}) \cup ((0, T] \times \partial (\Omega \times \Omega))$. 
The aim is to obtain an upper bound for $m$ using the fact that $u$ and $\bar u$ are viscosity solutions (and therefore sub- and supersolution to the corresponding equation).
Let $m \leq 0$ and assume first that the supremum in the second term is attained for $(\bar t, \bar x, \bar y) \in (0, T] \times \partial \Omega \times \Omega$, then by Lipschitz regularity of $u_i$
\begin{align*}
\sup_{\II \times Q^*}(\psi^i(t,x,y))^+ &\leq \sup_{(t, x) \in \partial^* Q_T} | u_i(t,x) - \bar u_i(t,x) | + [u_i(t,\cdot)]_1 |\bar x-\bar y| - \frac{1}{\delta}|\bar x-\bar y|^2 \\
& \leq \sup_{(t, x) \in \partial^* Q_T} | u_i(t,x) - \bar u_i(t,x) | + \frac{\delta}{4} ([u_i(t,\cdot)]_1)^2.
\end{align*}
{
Similar bounds can be obtained, using the Lipschitz regularity of the boundary and initial conditions in $\overline{\Omega}$, for any $(\bar t, \bar x, \bar y) \in Q^*$}. \\

Let $m > 0$ and consider that the supremum for $\bar m$ is attained at some point $(i_0, t_0, x_0, y_0)$. 
Since $m > 0$, arguing by contradiction, it follows that $(t_0, x_0, y_0) \not\in Q^*$, $\bar m > 0$ and by Lemma \ref{lem:Mishii}, the index $i_0$ may be chosen so that $\bar u_{i_0}(t_0, y_0) < \mathcal{M}_{i_0} \bar u_{i_0}(t_0, y_0)$.

The rest of the proof is identical to the proof of 
Theorem A.3 in \cite{barles2007error} and we only give a sketch. 
As $(t_0, x_0, y_0) \not\in Q^*$ and $i_0$ is chosen such that the equation holds at the maximum point for $\bar m$, we can apply the maximum principle as in Theorem \ref{thm:maxprinc} to $\bar m$ and use the resulting inequalities to obtain an upper bound for $m$.
Then, switching the roles of $u$ and $\bar{u}$ as super- and subsolution we obtain the lower bound.
\end{proof}



\subsection{Convergence rate for a switching system}
\label{sec:convSW}

Based on the regularity results from the previous section, we derive the convergence of the switching system to the HJB equation. 
To do so, we introduce three different second order non-linear parabolic equations (equations \eqref{Sw1}, \eqref{HJBi}, and \eqref{Sw2} below). 

First, we consider the following type of switching systems,
\begin{align}
\label{Sw1}
F_i(t,x,v,\partial_t v_i,Dv_i,D^2v_i)&=0 && \text{in}\quad Q_T, \quad
i\in\II:=\{1,\dots,M\},\\
v(0,x)&=\Psi_0(x) 		&&\text{in}\quad \overline{\Omega},\nonumber \\
v(t,x)&=\Psi_1(t, x)	&&\text{in}\quad(0, T] \times \partial \Omega,\nonumber
\end{align}
where the solution is $v=(v_{1}, \ldots ,v_{M})$, and for 
$i\in\II$, 
$(t,x)\in Q_T$, $r =(r_{1}, \ldots ,r_{M}) \in \R^M$, $p_t\in\R$, $p_x\in
\R^d$, and $X \in \mathcal{S}^d$, $F_{i}$ is given by
$$
F_i(t,x,r,p_t,p_x,X)
=\max\Big\{p_t+\sup_{\alpha \in \mathcal{A}_{i}}\mathcal{L}^{\alpha}(t,x,r_i,p_x,X);
r_i-\mathcal{M}_ir\Big\},$$
$\mathcal{A}_{i}$ is a subset of $\mathcal{A}$, $\mathcal{L}^{\alpha}$ is defined in \eqref{eq:def_L} and $\mathcal{M}_ir$ in \eqref{eq:defM}.

Our objective is to obtain a convergence rate for \eqref{Sw1}, as $k \to 0$ (the switching cost in \eqref{eq:defM}), to the following HJB equation
\begin{align}
\label{HJBi}
u_t+\sup_{\alpha \in \tilde \A} \mathcal{L}^{\alpha}(t, x,u,Du,D^2u)&=0 && \text{in}
\quad Q_T,\\
u(0,x)&=\Psi_0(x) 		&&\text{in}\quad \overline{\Omega},\nonumber \\
u(t,x)&=\Psi_1(t, x)		&&\text{in}\quad(0, T] \times \partial \Omega,\nonumber
\end{align}
where $ \tilde \A = \cup_{i}\, \A_{i}$.


The following proposition is a corollary of  Theorems \ref{thm:exist-unique} and \ref{thm:WPmain}.
\begin{proposition} 
\label{compSw1}
Assume (A1), (A2) and (A3) hold. 
Let $v$ and $u$ be the unique viscosity solutions of \eqref{Sw1} and \eqref{HJBi} respectively. 
Then,  $$|v|_1 +|u|_1 \leq \bar C,$$
where the constant $\bar C$ only depends on $T$, the constants  appearing in (A1), (A3) { and the bounds on $\zeta$, $\Psi_1$ and their derivatives}. 
\end{proposition}

To obtain the convergence rate we will use a regularization approach introduced by Krylov \cite{krylov2000}. Krylov's regularization procedure shows a way to construct smooth subsolutions on $Q_T$ by mollification of the solution to a system with ``shaken coefficients". 
For bounded domains, if applied directly in $\overline \Omega$, this requires to define such a solution at points lying outside the domain. To avoid this, the use of Krylov's technique has to be restricted to a smaller domain, whereas the estimates close to the boundary are obtained using \eqref{eq:barrierEstSW}.   

We define the auxiliary system
\begin{align}
\label{Sw2}
F^\eps_i(t,x,v^\eps,\partial_t v_i^\eps,Dv^\eps_i,D^2v^\eps_i)&=0 
 &&\text{in}\quad (0, T+\eps^2] \times \Omega, \quad i\in\II,\\
v^\eps_i(0,x) &= \Psi_0(x) 		&& \text{in} \quad \overline{\Omega},\nonumber \\
v^\eps_i(t,x) &= \Psi_1(t, x)			&& \text{in} \quad (0, T+\eps^2] \times \partial\Omega,\nonumber
\end{align}
where $v^\eps = (v^\eps_1, \cdots, v^\eps_M)$,
\begin{align*}
&F^\eps_i(t,x,r,p_t,p_x,X) = \\ 
&\qquad \max\left\{p_t+\sup_{ \substack{\alpha \in\A_{i} \\ 0 \leq \eta  \leq\eps^2, |\xi|\leq\eps}}
 	\mathcal{L}^{\alpha}\left( t+\eta,x+\xi,r_i,p_x,X\right);\ 
r_i-\mathcal{M}_ir \right\},
\end{align*} 
and the coefficients in the definition of $\mathcal{L}^{\alpha}$ in \eqref{eq:def_L} are extended to the relevant domain according to McShane's Theorem, Theorem \ref{mcshane}.


The proof of the main result in this section relies on the Lipschitz continuity in space of the solution to the family of switching systems with parameter $\e > 0$ in \eqref{Sw2}.

From assumption (A3), Theorems \ref{thm:exist-unique} and \ref{thm:A3} we infer the following result.
\begin{proposition} \label{WPi2}
Assume (A1), (A2) and (A3). There exists a unique continuous viscosity solution  $v^\e : [0, T+\e^2] \times \overline{\Omega} \to \mathbb{R}^M$ to \eqref{Sw2}.
Moreover, for all $i \in \II$
$$|v^\eps_i|_1 \leq \bar C \text{ in } [0, T+\eps^2] \times \overline{\Omega}
\quad \text{and} \quad
|v_i^\eps-v_i|_0 \leq \bar C\eps \text{ in } \overline{Q}_T,$$
where $v$ solves \eqref{Sw1} and  $\bar C$ only depends on $T$, the constants from (A1), (A3) { and the bounds on $\zeta$, $\Psi_1$ and their derivatives}.
\end{proposition}
\begin{proof}
The operator $F^\eps_i$  in \eqref{Sw2} is obtained from $F_i$ in \eqref{Sw1} by replacing the coefficients $\phi^\alpha=b,\sigma,c,f$  by $\phi^\alpha(\cdot+\eta,\cdot+\xi)$ for $0\leq \eta\leq \eps^2$ and $|\xi|\leq \eps$. Therefore, for $\eps$ small enough, assumptions (A1), (A2) and (A3) are still satisfied by \eqref{Sw2} with possibly different constants $C_0,C_1$ and a suitably modified function $\zeta$.
As a consequence, Theorem \ref{thm:exist-unique} holds true and $v^\eps$ satisfies \eqref{eq:barrierEstSW}.
The bound on $|v^\e_i|_1$ and the last claim follow respectively  by Theorem 
\ref{thm:WPmain}  and the continuous dependence estimates in Theorem \ref{thm:A3}. 
In particular, we have that $v^\eps_i=v_i$ on $\partial^*Q_T$  and for the coefficients  $\phi = b, \sigma, c, f$ 
\begin{align*}
\left|\phi^\alpha(t, x) - \phi^\alpha(t+\eta,x+\xi)\right| &\leq 2\e[\phi^\alpha]_1.
\end{align*}
\end{proof}





\label{eq:deftildev}

Using these results, we have all the necessary ingredients to state and prove the rate of convergence of \eqref{Sw1} to \eqref{HJBi}
 for the case of bounded spatial domains with Dirichlet boundary conditions.
 Part of the proof is very close to the one in Theorem 2.3 in \cite{barles2007error}, but they differ in  the way to estimate the bound close to $\partial\Omega$.

\begin{theorem}
\label{sw-rate}
Assume (A1)--(A3). 
If $u$ and $v$ are the solutions of \eqref{HJBi} and \eqref{Sw1} respectively, then for $k$ small enough, 
$$0\leq v_i - u\leq Ck^{1/3}\quad \text{in}\quad \overline Q_T,\quad i\in\II,$$
where $C$ only depends on $T$, the constants from (A1)--(A3) { and the bounds on $\zeta$, $\Psi_1$ and their derivatives}.
\end{theorem}

\begin{proof}
For the lower bound consider $w=(u,\dots,u) \in \mathbb{R}^M$. 
It is easy to check that $w$ is a subsolution of \eqref{Sw1}.
Then, given that $w = v_i$ on $\partial^* Q_T$ for $i\in\II$, 
by comparison for \eqref{Sw1} (Proposition~\ref{compSw1}) yields $u\leq
v_i$ for $i\in\II$. 

For the upper bound we use the regularization procedure of Krylov
\cite{krylov2000}.
Consider the system \eqref{Sw2} and let $v^\eps$ be its unique solution.
By the same arguments as in the proof of Theorem 2.3 in \cite{barles2007error}, shifting the variables preserves the subsolution property, in particular, ${v}^{\eps}(t-s,x-e)$ is a subsolution of the following system of independent equations
\begin{align}
\label{LSys}
\partial_t w_i+\sup_{\alpha \in\A_{i}}\mathcal{L}^\alpha (t,x,w_i,Dw_i,D^2w_i)= 0 \quad
\text{in}\quad  Q^\eps_T, \quad i\in\II,
\end{align}
where $Q^\e_T$ is the following restricted domain  
$$
Q^\e_T := (0, T]\times \Omega^\e
$$
with $\Omega^\e:=\{x\in \Omega : d(x)>\e\}.$
Next, we define ${{v}_\eps}:={v}^\eps*\rho_\eps$ where $\{\rho_\eps\}_{\eps}$ is the sequence of mollifiers defined in \eqref{eq:rhoxt} and conclude that ${v}_\eps$ is also a subsolution of equation \eqref{LSys}.
This is a consequence of using a Riemann-sum approximation to the mollification and using Lemma 2.7 in \cite{barles2002convergence}.

Moreover, one can verify that
$$ |{v}^{\eps}_i-{v}^{\eps}_j|_0\leq k, \quad i,j\in\II$$
and by properties of mollifiers and the previous bound, using integration by parts we obtain the same bounds as in \cite{barles2007error}
\begin{align*}
|\partial_t {v}_{\eps\,i}-\partial_t {v}_{\eps\,j}|_0\leq C\frac k {\eps^2},\quad
|D^n {v}_{\eps\,i} - D^n {v}_{\eps\,j}|_0\leq C\frac k {\eps^n}, \quad
n\in\N,\quad i,j\in\II,
\end{align*}
where $C$ depends only on $\rho$ and the constant $\bar C$ in Proposition \ref{WPi2}.
As a result, restricting the arguments in \cite{barles2007error} to the domain $Q^\eps_T$ and using Theorem \ref{thm:A3}, one has


\begin{align}\label{ineq1}{v}_{\eps\,i} - u \leq 
	e^{C t} \left( \sup_{\partial^* Q_T^\eps} |{v}_{\eps\,i}(t,x)-u(t,x)| +
 Ct\frac{k}{\eps^2} \right) \quad \text{in}\quad \overline{ Q_T^\eps}, \quad i\in\II,\end{align}
where as defined previously $\partial^* Q_T^\eps$ denotes the parabolic envelope of $Q_T^\eps$, i.e.\
$\partial^* Q_T^\eps = (\{\e^2\} \times \overline{\Omega} ) \cup ( (\e^2, T] \times \partial \Omega^\e)$.
Moreover by Proposition \ref{WPi2}, the regularity of $v$ and properties of mollifiers, we have  
\begin{align}\label{ineq1b}v_i - v_{\eps\,i}\leq  C\eps \quad 
\text{in}\quad \overline{Q_T}, \quad i\in\II.\end{align}



It remains to estimate $v_i-u$ in $\overline{Q_T}\setminus Q^\eps_T$. We use the regularity of $u$ and $v_i$ (Proposition \ref{compSw1}) together with the estimate \eqref{eq:barrierEstSW} (which holds true for both $u$ and $v_i$). One has 
\begin{align}\label{ineq2}
&|u(t,x)-v_i(t,x)|\leq ([u]_1+[v_i]_1) \eps \qquad \text{in}\quad
[0,\eps^2]\times\overline{\Omega}
\end{align}
and
\begin{align}\label{ineq3}
&|u(t,x)-v_i(t,x)|\leq K \zeta(t,x)\leq  K|D\zeta|_0 \eps \qquad \text{in}\quad
[0,T]\times (\overline{\Omega}\setminus\Omega^\e).
\end{align}
The result follows putting together inequalities \eqref{ineq1}, \eqref{ineq1b}, \eqref{ineq2}, \eqref{ineq3}, and minimizing with respect to $\e$.
\end{proof}

\section{Error bounds for discretizations of the Cauchy-Dirichlet problem}
\label{sec:Errors}

We start by listing our assumptions, which are the same as in \cite{barles2007error}, except (A5), but written here for bounded domains. 
For the HJB equation (\ref{E})--(\ref{eq:def_L}), in addition to (A1)--(A3) for $\mathcal{I}=\{1\}$ and setting $\mathcal{A} = \mathcal{A}_1$ we have:


\medskip
\noindent{\bf (A4)} The coefficients $\sigma^{\alpha}$, $b^{\alpha}$, $c^{\alpha}$, $f^{\alpha}$ are continuous in $\alpha$ for all $t$, $x$.

{
\noindent{\bf (A5)} $\Psi_1, \zeta \in C^2_\delta$ for $\delta\in (0,1]$ from assumption (S3)(ii) below.}

{
\begin{remark}
Assumption (A5) is necessary to estimate the scheme close to the boundary, see Lemma \ref{lem:barrier} below. Notice that this is also a natural requirement  to ensure that Assumptions 2.7 and 2.8 in \cite{dong2005rate} are satisfied, see Remark 2.9 there.
\end{remark}
}

%

\medskip
For the scheme \eqref{S} the following conditions need to be fulfilled.

\smallskip
\noindent {\bf (S1)} {\bf (Monotonicity)}  There exists $\lambda,
\mu\geq 0 ,h_0>0$ such that if $|h| \leq h_0$, $u\leq v$ are functions
in $C_b(\mathcal{G}_h)$, and $\phi(t)=e^{\mu t}(a+bt)+c$ for
$a,b,c\geq0$, then  
$$S(h,t,x,r+\phi(t),[u+\phi]_{t,x}) \geq  S(h,t,x,r,[v]_{t,x}) + b/2 -
\lambda c
\quad \text{in} \quad \mathcal{G}_h^+.
$$

\medskip
\noindent {\bf (S2)} {\bf (Regularity)}
For every $h>0$ and $\phi\in C_b(\G_h)$, the function 
$(t,x) \mapsto$\linebreak $S(h,t,x,\phi(t,x),[\phi]_{t,x})$
is bounded and continuous in $\G_h^+$ and the function $r \mapsto
S(h,t,x,r,[\phi]_{t,x})$ is locally uniformly continuous in $r$,
uniformly in $(t,x) \in \G_h^+$.

\medskip
{

\noindent {\bf (S3)} {\bf (Consistency)}
{\bf (i) }There exists a function $E$ 
such that for every $(t, x) \in \mathcal{G}^+_h \cap \Omega^{\e}$ (where $\Omega^\eps$ is defined in \eqref{eqdef:omegaeps}), 
$h=(\Delta t, \Delta x) > 0$,
and for any sequence $\{\phi_\eps\}_{\eps>0}$ of smooth functions satisfying
$$  |\partial_t^{\beta_0}D^{\beta'} \phi_\e (x,t) | \leq \tilde K
\e^{1-2\beta_0-|\beta'|}  \quad \hbox{in  }\overline Q_T ,\quad\text{ for
any $\beta_0\in\N_0$, $\beta'=(\beta'_i)_i\in\N_0^{d}$},$$
where $|\beta'|=\sum_{i=1}^d \beta_i'$, the following estimate holds:
\begin{align*}
\left| \partial_t \phi_{\e} + F(t, x, \phi_{\e}, D\phi_{\e}, D^2 \phi_{\e}) - S(h, t, x, \phi_{\e}(t, x), [\phi_{\e}]_{t,x}) \right| \leq E(\tilde K,h,\e).
\end{align*}
{\bf (ii) } There exists a function $\tilde E$ such that for any $(t, x) \in \mathcal{G}^+_h$, $h>0$, 
 and any $\phi\in C^{2}_\delta(\overline Q_T)$ for some $\delta\in (0,1]$, 
\begin{align*}
\left| \partial_t \phi + F(t, x, \phi, D\phi, D^2 \phi) - S(h, t, x, \phi(t, x), [\phi]_{t,x}) \right| \leq \tilde E(h,|\phi|_{2,\delta})
\end{align*}
and $\tilde E(h,\cdot)\to 0 $ as $h\to 0$.
\\
}
\medskip
\noindent {\bf (S4)} {\bf (Stability)} For every $h$ the scheme \eqref{S} has a unique solution in $C_b(\G_h)$. 
\medskip

\begin{remark}
Typical monotone approximation schemes considered in the literature are various finite
difference schemes (see, e.g., \cite{kushner2001numerical, BOZ04, oberman2010convergence}) and control schemes based on the dynamic programming principle
(see, e.g., \cite{camilli1995approximation, debrabant2013semi}). However, the problem being restricted to a domain with strong Dirichlet conditions, the scheme may require to be modified close to the boundary, see \cite{ReisingerRotaetxe17}.
\end{remark}

\begin{remark}
The consistency property (S3)(i) is introduced in \cite{barles2007error} and it is seen from the proof of the error bounds that the family of test functions for which it is stated constitutes the minimal requirement for the use of Krylov's mollification arguments.  
The definition of consistency in \cite{barles2007error} is slightly more general since it takes into account the case of different $E$ in the upper and lower bound. We do not allow for this here to simplify the notation and since it is not relevant in our application. Note that we require this property only in the interior $\Omega^\eps$ of the domain.
\end{remark}


\subsection{Discrete comparison result and estimates near boundaries}

We first state a comparison result for bounded continuous sub- and supersolutions of the numerical scheme \eqref{S} implied by assumptions (S1) and (S2).
This result is a slight modification of Lemma 3.2 in \cite{barles2007error}.
\begin{lemma}
\label{lem:SchemeComp}
Assume (S1), (S2), and that $u,v\in C_b(\G_h)$ satisfy
$$ S(h,t,x,u (t,x),[u]_{t,x})\leq g_1 \quad\hbox{in  } \G_h^+\; ,$$
$$ S(h,t,x,v (t,x),[v]_{t,x})\geq g_2 \quad\hbox{in  } \G_h^+\; ,$$
where $g_1, g_2 \in C_b(\G_h)$. Then
$$u-v\leq e^{\mu t} \sup_{(t, x) \in \partial^* \mathcal{G}_h} |(u(t,x)-v(t,x))^+|_0 + 2te^{\mu
  t}|(g_1-g_2)^+|_0,$$ 
where  $\mu$ is given by (S1).
\end{lemma}
\begin{proof}
The proof follows by Lemma 3.2 in \cite{barles2007error} once we have accounted for the difference of $u$ and $v$ at the parabolic boundary.
\end{proof}

As already noted in the proof of Theorem \ref{sw-rate}, Krylov's regularization procedure produces a smooth semi-solution only in the restricted domain $Q^\e_T$. As a consequence, the consistency property of the scheme cannot be used up to the parabolic boundary of the domain. 
The following result provides an important control on the solution of the scheme, $u_h$, and will allow us 
to obtain estimates in a neighbourhood of $\partial\Omega$.

\begin{lemma}\label{lem:barrier}
\label{lem:barrierSW}
Let assumptions (A1)--(A3), (A5) and (S1)--(S4) be satisfied and let $u_h$ be the solution of the scheme \eqref{S}. One has
\begin{equation}\label{eq:SbarrierEst}
\begin{split}
-K \zeta(t,x) - e^{\mu t}\sup_{\partial^*\G_h}|\big( &(\Psi_{1}-K\zeta)-u_h\big)^+|_0\leq  u_h(t,x)-\Psi_1(t,x)\\
& \leq K \zeta(t,x) + e^{\mu t}\sup_{\partial^*\G_h}|\big(u_h - (\Psi_{1}+K\zeta)\big)^+|_0 \quad \text{in} \quad \G^+_h,
\end{split} 
\end{equation}
for a  constant $K$ independent of $h$  and for $\mu$ given by (S1).
\end{lemma}
\begin{proof}
We are going to prove that there exists a suitable constant $K$ such that $g:=\Psi_{1} + K \zeta$ is a supersolution of the scheme \eqref{S} in $\G^+_h$. 
{
Since $g\in C^{2}_{\delta}(\overline Q_T)$ for $\delta\in (0,1]$, the consistency property (S3)(ii), together with assumptions (A2) and (A3), gives  (the constant $C$ below depends only on the constant $C_0$  in assumption (A1) and the bound on the derivatives of $\Psi_1$)
\begin{align*}
& S(h,t,x,g(t,x),[g]_{t,x}) \\
& \geq g_t + F(t,x,g,Dg,D^2 g) - \tilde E(h,|g|_{2,\delta})\\
& \geq K\zeta_t + \sup_{\alpha\in\mathcal A}\Big\{-K\,b^\alpha(t,x)D\zeta-K\,tr[a^\alpha(t,x)D^2\zeta] -Kc^\alpha(t,x)\zeta\Big\} -C-\tilde E(h,|g|_{2,\delta})\\
& \geq K - C-\tilde E(h,|\Psi_1|_{2,\delta}+K|\zeta|_{2,\delta})\geq 0
\end{align*}
for $K$ big enough and $h=h(K)$ sufficiently small. Therefore, from Lemma \ref{lem:SchemeComp} it follows that 
$$
u_h - g\leq e^{\mu t}\sup_{(t,x)\in\partial^*\G_h}|\big(u_h - g\big)^+|_0
$$
in $\overline Q_T$.  In the same way it is possible to show that $\Psi_1 - K\zeta$ is a subsolution and obtain the other inequality.}
\end{proof}

\subsection{Upper bound by Krylov regularization}

In this section, we prove an upper bound for the difference between the solution of \eqref{E}--\eqref{BC} and the numerical solution of the scheme \eqref{S}. The arguments follow \cite{KrylovDong07Estimates}. Before stating the result, we introduce the functions and  equations involved in the proof and give some preliminary results.

Let $\e\in (0,\e_0]$. 
We start by considering the solution $u^\e$ to a shaken equation:
\begin{align}
\label{eq:shakenE}
u^{\e}_t + 
\sup_{\substack{0 \leq \eta \leq \e^2, |\xi| \leq \e\\ \alpha\in\mathcal A}}\ \mathcal L^{\alpha}(t+\eta, x+\xi, u^\e,Du^{\e}, D^2 u^{\e}) = 0 & && \text{in} \quad (0, T+\e^2]\times\Omega, \\
\label{eq:shakenIC}
 u^{\eps}(0, x) = \Psi_0 (x) & && \text{in} \quad \overline{\Omega}, \\
\label{eq:shakenBC}
 u^{\eps}(t,x) 	= \Psi_1(t,x) & && \text{in} \quad (0,T+\varepsilon^2] \times \partial \Omega,
\end{align}
where for every $\alpha\in\mathcal A$, $(\eta,\xi)\in\R\times\R^d$ such that $0\leq \eta\leq \e^2, |\xi|\leq \e$, the new differential operator is obtained by $\mathcal L^\alpha$ replacing the coefficients $\phi^\alpha(t,x)$ for $\phi\equiv a,b,c,f$ by $\phi^\alpha(t+\eta, x+\xi)$
(when necessary, the coefficients are extended appropriately  by McShane's Theorem, Theorem \ref{mcshane}).



\begin{theorem}[\,{\bf Upper bound}\,]
\label{mainresA}
Assume (A1)--(A5) and (S1)--(S3). Let $u$ denote the solution of \eqref{E} satisfying \eqref{IC} and \eqref{BC} in the strong sense, and let $h$ be sufficiently small.
Then there exist constants $C$ {and  $\bar C$ (from Proposition \ref{WPi2})}
depending only on the constants in assumptions (A1)--(A3), (S1)--(S3) and the bounds on $\Psi_1$, $\zeta$ and their derivatives,  such that
\begin{align}
\label{eq:UpperB}
u-u_h \leq e^{\mu T} \sup_{(t, x) \in \partial^* \mathcal{G}_h}|(u -u_h)^+|_0   
				+ C\min_{\e>0} \left(\e + {
				E ( \bar C, h,\eps)}\right) 
\quad\text{in}\quad \mathcal{G}_h.
\end{align}
\end{theorem}
\begin{proof}
We start by considering $u^\e$, the unique viscosity solution to \eqref{eq:shakenE}--\eqref{eq:shakenBC}. 
As a special  case of Proposition \ref{WPi2}, we have that $u^\e$ is Lipschitz in space and H\"older continuous with exponent $\frac{1}{2}$ in time.

For any fixed $\eta = s$, $\xi = e$, with $0 \leq s \leq \e^2$ and $|e|\leq \eps$, $u^{\eps}(t+\e^2-s,x-e)$ is a subsolution of 
\begin{align*}
u_t + F(t, x, u(t,x), Du, D^2 u) = 0 \quad \text{in} \quad (0, T] \times \Omega^\eps.
\end{align*}

Now let 
$${u}_\e(t,x) \defeq  \int_{0 \leq s \leq \e^2} \int_{|e|\leq \e} 
	{u}^{\eps}(t+\e^2- s,x-e) \, \rho_\e(s, e) \, \diff e \, \diff s ,$$ 
where $\{\rho_\e\}_\e$ is the sequence of mollifiers defined in \eqref{eq:rhoxt}. 
Realizing that ${u}_\e$ is a convex combination of viscosity subsolutions and by stability results of viscosity solutions, see Lemma 2.7 in \cite{barles2002convergence}, we conclude that ${u}_\e$ is a classical subsolution of \eqref{E} in $(0, T] \times \Omega^\eps$. 

For all $(t,x) \in \overline Q_T$ one has
$$|u(t, x) - {u}_\e(t, x)| \leq |u(t, x) - {u}^\e(t, x)| + |{u}^\e(t, x) - {u}_\e(t, x)|,$$
where we can bound the second term using the regularity of $u^\eps$ and properties of the mollifier. 
To estimate the first term we employ Theorem \ref{thm:A3} and the fact that both functions are viscosity solutions of their  corresponding equations, therefore for any $(t,x) \in \overline Q_T$
$$e^{-\lambda t}|u(t,x)-{u}^\e (t,x)| \leq \sup_{\partial^* Q_T} |u - {u}^\e| + C \e = C \e,$$
where $C > 0$ depends on $T$, the $|\cdot|_1$ norm of the coefficients, but not on $\e$.
It follows by the properties of mollifiers that for any $\beta\in\N_0\times \N^d_0$ 
$$
|\partial_t^{\beta_0}D^{\beta'}u_\eps|_0\leq | u^\eps|_1 \eps^{1-2\beta_0-|\beta'|}\quad\text{and}\quad [\partial_t^{\beta_0}D^{\beta'}u_\eps]_{1}\leq |u^\eps|_1,
$$
where $| u^\eps|_1 \le \bar C$ by Proposition \ref{WPi2},
so that  by the consistency property (S3)(i) and the fact that ${u}_\e$ is a smooth subsolution, we have that
$$ S(h,t,x,{u}_\e (t,x),[{u}_\e]_{t,x})\leq E(\bar C,h,\eps) \quad\hbox{in  } \G^{\e,+}_h\; .$$
Finally, we compare $u_h$ and ${u}_\e$ using the scheme's comparison principle formulated in Lemma \ref{lem:SchemeComp}, and use it to establish the upper bound as
\begin{align*}
u - u_h &\leq e^{\mu T} \sup_{(t, x) \in \partial^* \mathcal{G}^\e_h}|({u}_\e -u_h)^+|_0   
				+ C\min_{\e>0} \left(\e +  E(\bar C,h,\eps)\right) \\
		&\leq e^{\mu T} \sup_{(t, x) \in \partial^* \mathcal{G}^{\e}_h}|(u -u_h)^+|_0   
				+ C\min_{\e>0} \left(\e +  E(\bar C,h,\eps)\right).
\end{align*}

It remains then to estimate $u-u_h$ in $\G_h\setminus \G^\e_h$, i.e.\ at points $(t,x)$ such that $t\in [0,T]$ and $x$ is in a neighbourhood of $\partial\Omega$. Applying Lemma \ref{lem:barrier}, on $\G_h\setminus\G^\e_h$ one has
$$
u-u_h \leq 2K\zeta +e^{\mu t}\sup_{\partial^*\G_h}|((\Psi_1-K\zeta)-u_h)^+|_0\leq 2K|D\zeta|_0\e +e^{\mu t}\sup_{\partial^*\G_h}|((\Psi_1-K\zeta)-u_h)^+|_0,
$$  
where the second inequality follows by the Lipschitz continuity of $\zeta$ and \eqref{eq:zeta1}. The proof is then concluded by observing that from \eqref{eq:zeta0} 
$$\sup_{\G^1_{h}}|((\Psi_1-K\zeta)-u_h)^+|_0\leq \sup_{\G^1_{h}}|(\Psi_1-u_h)^+|_0 = \sup_{\G^1_{h}} |(u-u_h)^+|_0 $$
and, by virtue of \eqref{eq:psi0psi1}, for $K$ big enough,
$$\sup_{\G^0_{h}}|((\Psi_1-K\zeta)-u_h)^+|_0\leq \sup_{\G^0_{h}}|(\Psi_0-u_h)^+|_0 = \sup_{\G^0_{h}} |(u-u_h)^+|_0. $$
 
\end{proof}

\subsection{Lower bound by switching system approximation}

For the derivation of the lower bound we follow the approach in \cite{barles2007error} and use a switching system approximation to build ``almost smooth" supersolutions to \eqref{E}.  
There are two main steps in the proof. First, we consider the case of a finite control set $\mathcal{A}$. 
Then the result is extended to the general case using assumption (A4).
It is in the first step that the proof needs to be adapted for the case of a bounded domain with Dirichlet conditions. 
The second part is identical to the original proof in \cite{barles2007error}.

The next set of lemmas contain some key results regarding the solutions of the auxiliary switching system below and its relation to the solution of \eqref{E}--\eqref{BC}.
The purpose of this auxiliary system is to ensure that the ``almost smooth supersolution" is defined for the whole $Q_T$.
\begin{align}
\label{Sw3}
F^\eps_i(t,x,v^\eps,\partial_t v^\eps_i,Dv^\eps_i,D^2v^\eps_i)&=0 &&
\text{in} \quad (0, T+ 2\e^2] \times \Omega, \ i\in\mathcal{I}=\{1,\dots,M\}, \\
v^\eps(0,x)&= \Psi_0 (x) && \text{in} \quad \overline{\Omega}, \nonumber \\ 
 v^{\eps}(t,x) 	&= \Psi_1(t,x)&& \text{in} \quad (0,T+ 2\e^2] \times \partial \Omega, \nonumber
\end{align}
where,
\begin{align}
& F^\eps_i(t,x,r,p_t,p_x,X)= \nonumber \\
\label{esc}
& \qquad\max\Big\{p_t+\min_{0\leq s \leq
\eps^2, |e|\leq\eps}\mathcal{L}^{\alpha_i}\left( t+\eta,x+\xi,r_i,p_x,X\right);\
r_i-\mathcal{M}_i r\Big\},
\end{align}
for any $\alpha_i \in \A$, $\mathcal{L}^{\alpha_i}$ is defined in \eqref{eq:def_L} and $\mathcal{M}_i r$ in \eqref{eq:defM}.
As a consequence of Proposition \ref{WPi2} and Theorem 
\ref{sw-rate} one has:
\begin{lemma}
\label{l:sw}
Assume (A1)--(A3), then the solution $v^\eps$ of \eqref{Sw3} satisfies
$$|v^\eps_i|_1\leq \bar C, \quad 
	|v^\eps_i - v^\eps_j|_0\leq k,\textit{and, for small k,} \quad
	\max_{i\in \II} |u- v_{i}^\eps|_0\leq C(\eps+k^{1/3}),
$$
where $u$ solves \eqref{E}--\eqref{BC} for $\mathcal A=\{\alpha_1,\ldots,\alpha_M\}$, $i,\,j\in \II$, and $\bar C$
{ (from Proposition \ref{WPi2})}
 and $C$ only depend on $T$, the constants from (A1)--(A3)  and the bounds on $\Psi_1$, $\zeta$ and their derivatives.
\end{lemma}

The next two lemmas  concern certain properties of the mollification of the solution $v^\e$. 
They are identical to Lemmas 3.4 and 3.5 in \cite{barles2007error}, restricted to points in $\Omega^\eps$, so that their proofs still hold.

\begin{lemma}
\label{vesupersol}
Assume (A1)--(A3) and $\eps \leq (8\sup_i[v^\eps_i]_{1})^{-1}k$ where $v^\eps$ is solution to \eqref{Sw3}. 
Let 
\begin{eqnarray}
\label{eq:veps}
v_{\eps i} \defeq v^\eps_i(\cdot + \e^2, \cdot) \ast \rho_\eps, \quad\text{for } i \in \II,
\end{eqnarray}
then, if $j \defeq \mathrm{argmin}_{i\in\II}v_{\eps i} (t,x)$, we have that 
\begin{equation*}
	\partial_t {v}_{\eps j} (t,x) +
	    \mathcal{L}^{\alpha_j}(t,x,{v}_{\eps j}(t,x),D {v}_{\eps j} (t,x), D^2 {v}_{\eps j}(t,x))  \geq 0\quad\text{in }(0,T]\times\Omega^\e.
\end{equation*}
\end{lemma}

\begin{lemma}
\label{wtsupersol} 
Assume (A1)--(A3) and $\eps \leq (8\sup_i[{v}^\eps_i]_{1})^{-1}k$ where ${v}^\eps$ is solution to  \eqref{Sw3}. 
Let ${v}_{\eps i}$ as in (\ref{eq:veps}).
Then the function $w \defeq \min_{i\in\II}{v}_{\eps i}$ is an approximate supersolution
of the scheme \eqref{S} in the sense that
$$S(h,t,x,w (t,x),[w]_{t,x})\geq - E (\bar C,h,\eps)
\quad\hbox{in }\G_h^{\eps,+},$$
with $\bar C$ from Lemma \ref{l:sw}. 
\end{lemma}

\begin{theorem}[\,{\bf Lower bound}\,]
\label{mainresB}
Assume (A1)--(A5), (S1)--(S3). 
Let $u$ denote  the solution of \eqref{E}--\eqref{BC}, and let $h$ be sufficiently small. 
 There exist constants $C$  and  $\bar C$ (from Proposition \ref{WPi2})
 depending only on the constants in (A1)--(A3) and (S1)--(S3)
 {and the bounds on $\Psi_1$, $\zeta$ and their derivatives,}
 such that
\begin{align}
\label{eq:LowerB}
u-u_h \geq -e^{\mu T} \sup_{(t, x) \in \partial^* \mathcal{G}_h}|(u -u_h)^-|_0   
	 - C\min_{\e>0} \left(\e^{1/3} +  E(\bar C, h,\e)\right) 
\quad\text{in } \mathcal{G}_h.
\end{align}
\end{theorem}

\begin{proof}
We proceed as in \cite{barles2007error} and first consider the case of finite control set $\mathcal{A}=\{\alpha_1, \ldots, \alpha_M\}$,
which allows us to
approximate the original problem \eqref{E}--\eqref{BC} by the solution of the switching system \eqref{Sw3}
as per Lemma \ref{l:sw}.

We intend to construct approximate smooth supersolutions of \eqref{E} and then use an analogue argument to that in Theorem \ref{mainresA} to derive the lower bound. 
From Lemma \ref{vesupersol}, $w=\min_{i\in\II}{v}_{\eps i}$ is a supersolution of \eqref{E} in $(0, T+2\e^2]\times\Omega^\e$.  
As a consequence, let $k= 8\sup_i[v^\eps_i]_{1} \eps$ and use Lemma \ref{wtsupersol} together with Lemma~\ref{lem:SchemeComp} to compare $u_h$ and $w$, obtaining
$$ u_h-w  \leq e^{\mu t}\sup_{(t, x) \in \partial^* \mathcal{G}^\e_h} |(u_h(t,x)-w(t,x))^+|_0 + 2te^{\mu
  t}E ( \bar C,h,\eps) \quad\hbox{in  }\G^\e_h. $$ 
Moreover, by Theorem \ref{sw-rate},
$$ |w - u|_0 \leq C(\eps+k+k^{1/3}),$$
and therefore, using the properties of mollifiers in $\G^\e_h$,
  
\begin{align}
\label{eq:uh-u}u_h- u  \leq e^{\mu t}\sup_{(t, x) \in \partial^* \mathcal{G}^\e_h} |(u_h(t,x)-u(t,x))^+|_0 + 2te^{\mu
  t}E (\bar C, h, \eps) + C(\eps+k+k^{1/3}) 
   \end{align}
for some constant $C$. We conclude the first step observing that (by Theorem \ref{thm:exist-unique}   and  Lemma \ref{lem:barrierSW})
 in $[0,T]\times(\overline{\Omega}\setminus\Omega^\e)$ the following holds 
$$
u_h-u \leq 2K\zeta +e^{\mu T}\sup_{\partial^*\G_h}|(u_h-(\Psi_1-K\zeta))^+|_0\leq 2K{|D\zeta|_0}\e +e^{\mu T}\sup_{\partial^*\G_h}|(u-u_h)^{-}|_0,
$$  
and  minimizing \eqref{eq:uh-u} with respect to $\eps$.

The extension to general, not necessarily finite $\mathcal{A}$ is identical to \cite{barles2007error}. 
%
%
\end{proof}

\section{Error bounds for some monotone finite difference schemes}
\label{sec:ErrExamples}

In this section, we employ Theorems \ref{mainresA} and \ref{mainresB} to derive error bounds for two examples of finite difference schemes approximating \eqref{E}--\eqref{BC}.
The analysis for other monotone, stable and consistent schemes would follow the same steps.

\subsection{The scheme by Kushner and Dupuis}
\label{subsec:kushnerdupuis}

The scheme proposed in \cite[Section 5.3.1]{kushner2001numerical}) is based on a
seven-point stencil approximation of the second derivative terms (see also \cite[Section 5.1.4]{hackbush1992elliptic}), taking into account the sign of the off-diagonal term in
the diffusion matrix (the covariance).
It is of second order accurate and local in the sense that it only 
uses a node's immediate neighbours, and therefore it does not ``overstep'' the domain and can be used up to the boundary.
However,
it is only monotone if the diffusion matrix is strictly diagonally dominant. We assume in the following that this is the case.


The error bounds for this scheme were also analysed in Section 4 of \cite{barles2007error} for $\Omega = \R^d$. 
It follows directly from Theorems \ref{mainresA} and \ref{mainresB}
that the error bounds for this scheme applied to the Cauchy-Dirichlet problem are identical to the ones obtained in Theorem 4.1 in \cite{barles2007error}.

\begin{theorem}
If $u_h \in C_b(\mathcal{G}_h)$ is the solution of the Kushner and Dupuis scheme
 and $u$ is the solution of \eqref{E}--\eqref{BC}, then there is a $C > 0$ such that in $\mathcal{G}_h$
$$-e^{\mu t}\sup_{(t, x) \in \partial^* \mathcal{G}_h} |(u-u_h)^-|_0 - C h_\ell 
	\leq u- u_h  \leq 
e^{\mu t}\sup_{(t, x) \in \partial^* \mathcal{G}_h} |(u-u_h)^+|_0 + C h_u,$$
where $h_\ell$ and $h_u$ are defined as follows
$$h_\ell = \max(\dt^{1/10},\dx^{1/5}),\text{ and} \qquad h_u = \max(\dt^{1/4},\dx^{1/2}).$$
\end{theorem}

\subsection{The truncated semi-Lagrangian scheme}

We study now the error bounds for the truncated linear interpolation semi-Lagrangian (LISL) scheme defined in \cite{ReisingerRotaetxe17}.
The scheme 
is identical to Scheme 2 in \cite{debrabant2013semi} in a subset of the interior of the domain where the stencil does not ``overstep'' the boundary,
but is modified near the boundary to construct a consistent approximation using the boundary data.
It can be written 
in the form 
\begin{align}
\label{eq:LISLscheme}
S(h,t_n,x_j,r,[U]_{t_n,x_j}\!) = \max_{\alpha \in \mathcal{A}} \left\lbrace  S^\alpha(h,t_n,x_j,r,[U]_{t_n,x_j})   \right\rbrace,
\end{align}
$j=1,\ldots, J$, where $J$ is the number of mesh points and
\begin{align*}
S^\alpha(h,t_n,x_j,r,[U]_{t_n,x_j}\!) \defeq \mathcal{W}^{\alpha, n,n}_{j, j} r -\!  
\sum_{ i \neq j} \mathcal{W}^{\alpha, n,n}_{j, i} U^n_i \!-\! \sum_{i=1}^N  \mathcal{W}^{\alpha, n,n-1}_{j, i} U^{n-1}_i  \!\! - \! F^{\alpha, n - 1 + \theta}_j,
\end{align*}
where the coefficients $\mathcal{W}$ are given 
in \cite{ReisingerRotaetxe17}, Section 2.1 and 2.2, 
and depend on a parameter $\theta \in [0,1]$, where $\theta=0$ corresponds to the forward Euler scheme,
$\theta=1$ to the backward Euler scheme, and $\theta=1/2$ to the Crank-Nicolson scheme.
If
\begin{eqnarray}
\label{CFL}
\Delta t \le C (1-\theta) \dx^{3/2}
\end{eqnarray}
for a sufficiently small constant $C$, which can be explicitly given in terms of the data (see Proposition 2.4 and Corollary 2.5 in \cite{ReisingerRotaetxe17}),
the scheme is of positive type, i.e., $\mathcal{W}^{\alpha, n,m}_{j, i}\ge 0$ for all $1\le n\le N$, $m\in \{n-1,n\}$, $1\le i,j\le J$, and all $\alpha \in \mathcal{A}$.


\begin{proposition}
Under condition \eqref{CFL}, the solution of the scheme \eqref{eq:LISLscheme} is unique and satisfies assumptions (S1) and (S2).
If, moreover, $\sqrt{\dx}/\varepsilon \rightarrow 0$ for $\dx \rightarrow 0$, the scheme also satisfies (S3), where for $\dx$ sufficiently small,
$h=(\Delta t, \dx)$,
\begin{align}	
\label{eq:ELISL}
E(\tilde{K}, h,\e) \defeq \tilde{C} \, \tilde{K} \, (|1 -2\theta| \dt\eps^{-3} + \dt^2 \eps^{-5} +  \dx \eps^{-3}),
\end{align}
and $\tilde{C} > 0$ is a sufficiently large constant. 
\end{proposition}
\begin{proof}
To check (S1), we recall the positivity of the coefficients $\mathcal{W}$ under (\ref{CFL}).  It follows from a straightforward calculation that the scheme \eqref{eq:LISLscheme} is monotone.

(S2) follows from the continuity of the coefficients. 

Regarding the consistency (S3), for any smooth function $\phi$, for some $\kappa > 0$
\begin{align*}
\left| \partial_t \phi + F(t, x, \phi, D\phi, D^2 \phi) - S(h, t, x, \phi(t, x), [\phi]_{t,x}) \right| &\leq& \\ 
&& \hspace{-6.5 cm} 
\frac{|1 - 2\theta|}{2} \dt |\partial^2_t \phi|_0  + \kappa \dt^2 |\partial^{3}_t \phi|_0 + 
\kappa \left\{
\begin{array}{ll}
 \dx |D^4_x \phi|_0, & d(x) > \kappa \sqrt{\dx}, \\
\sqrt{\dx} |D^3_x \phi|_0, & d(x) \le \kappa \sqrt{\dx},
\end{array}
\right.
\end{align*}
with $d$ the distance function to $\partial\Omega$,
see the proof of Corollary 2.3 in \cite{ReisingerRotaetxe17}
 for more details of this calculation by Taylor expansion.

The final expression in (S3)(i) is obtained observing that $\Omega_\eps \subset \{x: d(x)>\kappa \sqrt{\dx} \}$ for $\dx$ small enough, and by the assumption that
$\partial^{k_1}_t D^{k_2}_x \phi = \mathcal{O}(\e^{1-2k_1-k_2})$ for any $k_1, k_2 \in \mathbb{N}_0$.
Similarly, (S3)(ii) follows directly.

The existence and uniqueness of the solution can be proved 
by adapting the arguments in Theorem 4.2 in \cite{debrabant2013semi} to \eqref{eq:LISLscheme}
or using the arguments in \cite{bokanowski_Howard}: 
by Theorem 2.1 in \cite{bokanowski_Howard}, there exists a unique solution to the system and the approximate solution computed using policy iteration converges to it.
\end{proof}

\begin{corollary}
\label{cor-SL}
Let $u$ be the solution of \eqref{E}--\eqref{BC} and $u_h$ the solution of its truncated LISL discretization.
Then, {under condition (\ref{CFL})}, there is $C > 0$ such that in $\mathcal{G}_h$
$$-e^{\mu t}\sup_{(t, x) \in \partial^* \mathcal{G}_h} |(u-u_h)^-|_0 - C h_\ell 
	\leq u- u_h  \leq 
e^{\mu t}\sup_{(t, x) \in \partial^* \mathcal{G}_h} |(u-u_h)^+|_0 + C h_u,$$
where $h_\ell$ and $h_u$ are defined by
$$h_\ell = \max(\dt^{1/10},\dx^{{1/10}}) \text{ and} \qquad h_u = \max(\dt^{1/4},\dx^{{1/4}})$$
{for $\theta\neq 1/2$, and by 
$$h_\ell = \max(\dt^{1/8},\dx^{{1/10}}) \text{ and} \qquad h_u = \max(\dt^{1/3},\dx^{{1/4}})$$
for $\theta=1/2$.}
\end{corollary}

\begin{proof}
Combining the bounds from Theorem \ref{mainresA} and Theorem \ref{mainresB} together with the consistency error given by \eqref{eq:ELISL}, the result follows by minimizing with respect to $\eps$ the following functions:
\begin{enumerate}
\item For the upper bound, by \eqref{eq:UpperB} and \eqref{eq:ELISL},
$$ \eps + C({|1-2\theta|}\dt\eps^{-3} + \dt^2\eps^{-5}+  \dx \eps^{-3}).$$
\item For the lower bound, by \eqref{eq:LowerB} and \eqref{eq:ELISL},
$$ \eps^{1/3} + C({|1-2\theta|}\dt\eps^{-3} + \dt^2\eps^{-5} + \dx \eps^{-3}).$$
\end{enumerate}

{If $\theta\neq 1/2$,
the minimum is attained by
$\eps=\max(\dt^{1/4},\dx^{{1/4}})$ in the first case and
$\eps^{1/3}=\max(\dt^{1/10},\dx^{{1/10}})$ in the second case;
if  $\theta = 1/2$, then
by $\eps=\max(\dt^{1/3},\dx^{{1/4}})$ in the first case and
$\eps^{1/3}=\max(\dt^{1/8},\dx^{{1/10}})$ in the second case.} 
\end{proof}

\begin{remark}[Optimal refinement strategy]
{
{ The leading errors in $\Delta x$ and $\Delta t$ are balanced for $\theta\neq 1/2$ by choosing} $\Delta x\sim\Delta t$, and for $\theta =1/2$
by $\Delta t\sim \Delta x^{3/4}$ in the upper bound and  
$\Delta t\sim \Delta x^{4/5}$ in the lower bound.
These choices do not satisfy the stability condition given by \eqref{CFL}.
It follows that the computationally optimal refinement can only be achieved for the {backward} Euler scheme {($\theta=1$)}, i.e.\ in absence of the CFL condition \eqref{CFL}, {while the higher order consistency in time of the case $\theta=1/2$ cannot be utilized.}}
\end{remark}
%

\begin{remark}[Relation to known results for $\Omega = \mathbb{R}^d$]
The upper bound in Corollary \ref{cor-SL} could be slightly improved by optimising the truncation error of the mollified solution over a generalised step-size $k$ instead
of fixing it to $\sqrt{\Delta x}$, which is optimal for sufficiently regular functions. This would reproduce the bound obtained in \cite{debrabant2013semi} for $\Omega = \mathbb{R}^d$.
However, a symmetric lower bound as in \cite{debrabant2013semi} requires a continuous dependence estimate for the scheme, which does not follow directly for the truncated scheme, even given the stronger assumptions in \cite{debrabant2013semi}.
Our lower bound is therefore similar to the one obtained for $\Omega = \mathbb{R}^d$ in \cite{SLproceedings} by a switching system approximation.
\end{remark}

\begin{remark}[Consistency at the boundary] 

\begin{enumerate}[(i)]
\item
In \cite{picarelli2017boundary}, a mesh refinement strategy in a layer near the boundary is proposed, which guarantees that the consistency error is globally of
order $\dx$.
The present analysis shows that the consistency order close to the boundary is irrelevant for the global error bounds, as the barrier function and not the truncation error is used
to control the error near the boundary.
This is consistent with the numerical tests in \cite{picarelli2017boundary} which show that the observed accuracy is largely unaffected by the refinement.
\item
The consistency of the scheme up to the boundary (without an algebraic order) is used to construct sub- and supersolutions in the proof of Lemma \ref{lem:barrierSW}.
The convergence analysis here does therefore not extend directly to schemes where consistency is lost near the boundary, such as constant extrapolation of the boundary
data (see \cite{picarelli2017boundary}) or cropping of the stencil to avoid stepping outside the domain as in \cite{feng2017convergent}.
\end{enumerate}
\end{remark}

\begin{remark}[Other monotone schemes]
The error analysis in this section is directly applicable to the hybrid scheme proposed in \cite{ma2016unconditionally}, where the local higher-order scheme from
\cite{kushner2001numerical} and Section \ref{subsec:kushnerdupuis}
is used for nodes where it leads to a positive coefficients discretisation, and a low-order wide stencil scheme similar to the semi-Lagrangian scheme from this section is used otherwise, to guarantee monotonicity. The order obtained (both theoretically and practically) is generally that of the low-order scheme.
Section 5.2 in \cite{ma2016unconditionally} also provides a detailed discussion of the consistency of the truncated scheme at the boundary in degenerate situations,
a setting not covered by our results.
\end{remark}

\section{Conclusion}
\label{sec:conc_err}

This paper extends the analysis in \cite{barles2007error} from the Cauchy problem  for HJB equations in $\R^d$ to the Cauchy-Dirichlet problem in bounded domains.
We use regularity conditions for the domain and the boundary data identical to \cite{KrylovDong07Estimates}. In particular, we adopt the use of a so-called barrier function
to control the regularity and error near the boundary.
Using the framework developed, we are able to analyse the classical Kushner and Dupuis scheme and the truncated semi-Lagrangian scheme proposed in \cite{ReisingerRotaetxe17}.
The error bounds obtained are of the same order as the known results from \cite{barles2007error} on the whole space, despite the lower consistency order of the semi-Lagrangian scheme near the boundary. 


%

{

The order of the lower bounds is not as good as in \cite{KrylovDong07Estimates}, as we consider more general, practically applicable fully discrete schemes. Continuous dependence estimates for such schemes, which are essential for symmetric error bounds with the current methodology, are still an open question on the whole space and for domains.

We did not present any numerical experiments as they are available in the literature for both semi-Lagrangian and finite difference schemes.
We point especially to \cite{ReisingerRotaetxe17} for detailed tests of the truncated semi-Lagrangian scheme.

Further work should include the relaxation of the regularity requirements on the boundary. At present, the existence of a regular barrier function encapsulates
the smoothness of the boundary and the non-degeneracy of the operator in relation to the boundary. While this allows for some examples with non-smooth boundaries (e.g., corners with small angles) and degeneracies (e.g., parallel to the boundary), as pointed out in Example 2.3 in \cite{KrylovDong07Estimates}, it rules out other simple examples.
It would be desirable to include all cases where Dirichlet data are satisfied strongly, or indeed to extend the schemes and their analysis to the general situation where boundary conditions are satisfied weakly.

 }

\bigskip

{\bf Acknowledgements}: We are grateful to G.~Barles, P.A.~Forsyth, and E.R.~Jakobsen for helpful comments and for pointing us to key references.

\appendix

\section{Background definitions and results}

\begin{definition}[\,{\bf Parabolic semijets, Definition 2.1 in \cite{JakobsenEstimates}}]\label{def:semijets}
For a function $u$ belonging to $\mathrm{USC}([0, T] \times \bar{\Omega}; \R)$ ($\mathrm{LSC}([0, T] \times \bar{\Omega}; \R)$) that is locally bounded, the second-order parabolic superjet (subjet) of $u$ at $(t, x) \in Q_T$, denoted by $\mathcal{P}^{2, +(-)}u(t, x)$, is defined as the set of triples $(a, p, X) \in \R \times \R^d \times \mathcal{S}^d$ such that
\begin{align*}
u(s, y) \!
\begin{array}{c}
\leq \\
(\geq) 
\end{array}
\!
u(t, x)  \!\!\;+ \!\!\; a(s\!\!\;-\!\!\;t) \!\!\;+\!\!\; \langle p, y \!\!\;- \!\!\;x\rangle  \!\!\;+ \!\!\; \frac{1}{2} \langle X(y\!\!\;-\!\!\;x), y\!\!\;-\!\!\;x\rangle 
+ o(|s\!\!\;-\!\!\;t| + |x\!\!\;-\!\!\;y|^2),
\end{align*}
as $Q_T \ni (s, y) \to (t, x)$. 
We define the closure $\overline{\mathcal{P}}^{2, +(-)}u(t, x)$ as the set of $(a, p, X) \in \R \times \R^d \times \mathcal{S}^d$ for which there exists $(t_n, x_n, a_n, p_n, X_n) \in Q_T \times  \R \times \R^d \times \mathcal{S}^d$ 
such that $(t_n, x_n, u(t_n, x_n), a_n, p_n, X_n) \to (t, x, u(t, x), a, p, X)$ as $n \to \infty$ and $(a_n, p_n, X_n) \in \mathcal{P}^{2, +(-)}u(t_n, x_n)$ for all $n$.
\end{definition}

\begin{theorem}[\,{\bf Crandall-Ishii lemma, Theorem 8.3 in \cite{crandall1992user}}\,]
\label{thm:maxprinc}
Let $u_1$ and $-u_2$ belong to $\mathrm{USC}({Q}_T;\R)$. Let $\phi \in C^{1, 2, 2}((0, T) \times \Omega	\times {\Omega})$, i.e.\ once continuously differentiable in $t \in (0, T)$ and twice continuously differentiable in $(x, y)\in {\Omega} 	\times {\Omega}$.
Suppose $(t_\phi, x_\phi, y_\phi) \in (0, T) \times \Omega \times \Omega$ is a local maximum of the function
$$(t, x, y) \to u_1(t, x) - u_2(t, y) - \phi(t, x, y).$$
Suppose that there is an $r >0$ such that for every $M>0$ there is a constant $C$ such that
\[
\left\{ 
\begin{array}{l}
	a \leq C \text{ whenever } (a, p, X)\in \mathcal{P}^{2, +}u_1(t, x),\\
    \qquad |x - x_\phi| + |t - t_\phi| \leq r, \,  |u_1(t, x)| + |p| + \|X\| \leq M, \\
	b \geq C \text{ whenever } (b, q, Y)\in \mathcal{P}^{2, -}u_2(t, y),\\
    \qquad |y - y_\phi| + |t - t_\phi| \leq r, \,  |u_2(t, y)| + |q| + \|Y\| \leq M. \\
                \end{array}
              \right.
\]
Then for any $\kappa > 0$ there exist two numbers $a, b \in \mathbb{R}$ and two matrices $X, Y \in \mathcal{S}^d$ such that
\begin{align*}
(a, D_x \phi(t_\phi, x_\phi, y_\phi), X) &\in \overline{\mathcal{P}}^{2, +}u_1(t_\phi, x_\phi), \\
(b, -D_y \phi(t_\phi, x_\phi, y_\phi), Y) &\in \overline{\mathcal{P}}^{2, -}u_2(t_\phi, y_\phi),
\end{align*}
\begin{align}\label{eq:ineqXY}
-\left(\frac{1}{\kappa} + \|A\|\right) \begin{pmatrix}
I & 0 \\
0 & I
\end{pmatrix} \leq \begin{pmatrix}
X & 0 \\
0 & -Y
\end{pmatrix} \leq A + \kappa A^2,
\end{align}
where $A = D^2 \phi(t_\phi, x_\phi, y_\phi)$, $\|A\|= \sup_{|\xi| \leq 1} \{\xi^\top A \xi  \}$ and $a - b = \partial_t \phi(t_\phi, x_\phi, y_\phi)$.
\end{theorem}

\begin{lemma}[\,{\bf Lemma V.7.1 in \cite{fleming2006controlled}}\,]
\label{lm:defModLi}
Let $\mathcal{L}^{\alpha}_i$ be as in \eqref{eq:defLi}. 
Assume (A1). Then, there exists a continuous function $\omega : [0, \infty) \to [0, \infty)$, independent of $i$, satisfying $\omega(0^+) = 0$ such that
\begin{align}
\label{eq:defModLi}
\sup_{\alpha \in \A_i}\mathcal{L}^{\alpha}_i(t,y,r,\beta(x-y),Y) - \sup_{\alpha \in \A_i}\mathcal{L}^{\alpha}_i(t,x,r,\beta(x-y),X)
	\leq \omega(\beta |x-y|^2 + |x-y|),
\end{align}
for every $(t, x), (t, y) \in Q_T$, $\beta > 0$, and symmetric matrices $X, Y \in \mathcal{S}^d$ satisfying 
\begin{align} \label{ineqXY}
-3 \beta \begin{pmatrix}
I & 0 \\
0 & I
\end{pmatrix}  \leq \begin{pmatrix}
X & 0 \\
0 & -Y
\end{pmatrix} \leq 3 \beta \begin{pmatrix}
I & -I \\
-I & I
\end{pmatrix}.
\end{align}

\end{lemma}

\begin{lemma}[\,{\bf Lemma A.2 in \cite{barles2007error}}\,]
\label{lem:Mishii}
Let $u\in USC(Q_T;\R^M)$ be a bounded above subsolution of
\eqref{eq:A1} and
$\bar{u}\in LSC(Q_T;\R^M)$ be a bounded below supersolution
of another equation of the form \eqref{eq:A1} where the functions
$\mathcal{L}^{\alpha}_i$ are replaced by functions
$\bar{\mathcal{L}}^{\alpha}_i$ satisfying the same
assumptions. Let $\phi:(0,T) \times \Omega \times \Omega\to \R$ be a smooth function
bounded from below. We denote
$$
\psi_i(t,x,y) \defeq u_i(t,x)-\bar{u}_i(t,y)-\phi(t,x,y),
$$ 
and $\bar{M}\defeq\sup_{i,t,x,y}\,\psi_i(t,x,y)$. 
If there exists a maximum point for $\bar{M}$, i.e.\ a point $(i',t_0,x_0,y_0) \in \II \times (0, T) \times \Omega \times \Omega$
such that $\bar{M} = \psi_{i'}(t_0,x_0,y_0)$, then there exists $i_0 \in \II$
such that $(i_0,t_0,x_0,y_0)$ is also a maximum point for $\bar{M}$ and 
$\bar{u}_{i_0}(t_0,y_0)<\mathcal{M}_{i_0}\bar{u}(t_0,y_0)$.
\end{lemma}

\begin{theorem}[\,{\bf McShane's Theorem. Corollary 1 in \cite{mcshane1934}}\,]
\label{mcshane}
If $f$ is a real function defined on a subset $E$ of a metric space $S$, and $f$ satisfies on $E$ a Lipschitz or H\"{o}lder condition
$$| f(x_1) - f(x_2)| \leq L d_E( x_1, x_2 ) ^\alpha,$$
where $\alpha \in (0, 1]$, then $f$ can be extended to $S$ preserving the Lipschitz or H\"{o}lder condition with the same constant $L$.
\end{theorem}

\bibliography{refs}        

\begin{thebibliography}{10}

\bibitem{barles2002convergence}
G.~Barles and E.~R. Jakobsen.
\newblock On the convergence rate of approximation schemes for
  {H}amilton--{J}acobi--{B}ellman equations.
\newblock {\em ESAIM: Mathematical Modelling and Numerical Analysis},
  36(1):33--54, 2002.

\bibitem{barles2007error}
G.~Barles and E.~R. Jakobsen.
\newblock Error bounds for monotone approximation schemes for parabolic
  {H}amilton--{J}acobi--{B}ellman equations.
\newblock {\em Mathematics of Computation}, 76(260):1861--1893, 2007.

\bibitem{barles1998strong}
G.~Barles and E.~Rouy.
\newblock A strong comparison result for the {B}ellman equation arising in
  stochastic exit time control problems and its applications.
\newblock {\em Communications in Partial Differential Equations},
  23(11-12):552--562, 1998.

\bibitem{barles1991convergence}
G.~Barles and P.~E. Souganidis.
\newblock Convergence of approximation schemes for fully nonlinear second order
  equations.
\newblock {\em Asymptotic Analysis}, 4(3):271--283, 1991.

\bibitem{bokanowski_Howard}
O.~Bokanowski, S.~Maroso, and H.~Zidani.
\newblock {Some convergence results for Howard's algorithm}.
\newblock {\em SIAM Journal on Numerical Analysis}, 47(4):3001--3026, 2009.

\bibitem{Picarelli2015dynamic}
O.~Bokanowski, A.~Picarelli, and H.~Zidani.
\newblock Dynamic programming and error estimates for stochastic control
  problems with maximum cost.
\newblock {\em Applied Mathematics \& Optimization}, 71(1):125--163, 2015.

\bibitem{BOZ04}
J.~Bonnans, E.~Ottenwaelter, and H.~Zidani.
\newblock {N}umerical schemes for the two dimensional second-order {HJB}
  equation.
\newblock {\em ESAIM: Mathematical Modelling and Numerical Analysis},
  38:723--735, 2004.

\bibitem{caffarelli2008Rate}
L.~A. Caffarelli and P.~E. Souganidis.
\newblock A rate of convergence for monotone finite difference approximations
  to fully nonlinear uniformly elliptic {PDE}s.
\newblock {\em Communications on Pure and Applied Mathematics}, 61(1):1--17,
  2008.

\bibitem{camilli1995approximation}
F.~Camilli and M.~Falcone.
\newblock An approximation scheme for the optimal control of diffusion
  processes.
\newblock {\em ESAIM: Mathematical Modelling and Numerical Analysis},
  29(1):97--122, 1995.

\bibitem{chaumont2004uniqueness}
S.~Chaumont.
\newblock Uniqueness to elliptic and parabolic {H}amilton--{J}acobi--{B}ellman
  equations with non-smooth boundary.
\newblock {\em Comptes Rendus Mathematique}, 339(8):555--560, 2004.

\bibitem{crandallIshii}
M.~G. Crandall and H.~Ishii.
\newblock The maximum principle for semicontinuous functions.
\newblock {\em Differential Integral Equations}, 3(6):1001--1014, 1990.

\bibitem{crandall1992user}
M.~G. Crandall, H.~Ishii, and P.-L. Lions.
\newblock User's guide to viscosity solutions of second order partial
  differential equations.
\newblock {\em Bulletin of the American Mathematical Society}, 27(1):1--67,
  1992.

\bibitem{debrabant2013semi}
K.~Debrabant and E.~R. Jakobsen.
\newblock Semi-{L}agrangian schemes for linear and fully non-linear diffusion
  equations.
\newblock {\em Mathematics of Computation}, 82(283):1433--1462, 2013.

\bibitem{SLproceedings}
K.~Debrabant and E.~R. Jakobsen.
\newblock Semi-{L}agrangian schemes for linear and fully non-linear
  {H}amilton--{J}acobi--{B}ellman equations.
\newblock In F.~Ancona, A.~Bressan, P.~Marcati, and A.~Marson, editors, {\em
  Hyperbolic {P}roblems: {T}heory, {N}umerics, {A}pplications}, volume~8 of
  {\em {AIMS} {S}eries on {A}pplied {M}athematics}, pages 483--490,
  {S}pringfield, {MO}, 2014. American {I}nstitute of {M}athematical {S}ciences
  ({AIMS}).

\bibitem{dong2005rate}
H.~Dong and N.~V. Krylov.
\newblock On the rate of convergence of finite-difference approximations for
  {B}ellman equations with constant coefficients.
\newblock {\em St. Petersburg Mathematical Journal}, 17(2):108--132, 2005.

\bibitem{KrylovDong07Regular}
H.~Dong and N.~V. Krylov.
\newblock On time-inhomogeneous controlled diffusion processes in domains.
\newblock {\em The Annals of Applied Probability}, 35(1):206--227, 2007.

\bibitem{KrylovDong07Estimates}
H.~Dong and N.~V. Krylov.
\newblock The rate of convergence of finite-difference approximations for
  parabolic {B}ellman equations with {L}ipschitz coefficients in cylindrical
  domains.
\newblock {\em Applied Mathematics \& Optimization}, 56:37--66, 2007.

\bibitem{falcone2013book}
M.~Falcone and R.~Ferretti.
\newblock {\em Semi-{L}agrangian Approximation Schemes for Linear and
  {H}amilton--{J}acobi Equations}.
\newblock Society for Industrial and Applied Mathematics, Philadelphia, PA,
  2013.

\bibitem{feng2017convergent}
X.~Feng and M.~Jensen.
\newblock Convergent semi-{L}agrangian methods for the {M}onge--{A}mp{\`e}re
  equation on unstructured grids.
\newblock {\em SIAM Journal on Numerical Analysis}, 55(2):691--712, 2017.

\bibitem{fleming2006controlled}
W.~H. Fleming and H.~M. Soner.
\newblock {\em Controlled Markov Processes and Viscosity Solutions}, volume~25.
\newblock Springer Science \& Business Media, 2006.

\bibitem{freidlin1985functional}
M.~I. Freidlin.
\newblock {\em Functional Integration and Partial Differential Equations},
  volume 109 of {\em Annals of Mathematics Studies}.
\newblock Princeton University Press, 1985.

\bibitem{friedman1975stochastic_v2}
A.~Friedman.
\newblock {\em Stochastic {D}ifferential {E}quations and {A}pplications}.
\newblock Number v. 2 in Probability and mathematical statistics. Academic
  Press, 1975.

\bibitem{hackbush1992elliptic}
W.~Hackbush.
\newblock {\em Elliptic Differential Equations: Theory and Numerical
  Treatment}.
\newblock Springer, 1992.

\bibitem{ishiiKoikeSwitching}
H.~Ishii and S.~Koike.
\newblock Viscosity solutions of a system of nonlinear second-order elliptic
  {PDE}s arising in switching games.
\newblock {\em Funkcialaj Ekvacioj}, 34(1):143--155, 1991.

\bibitem{jakobsen2010monotone}
E.~R. Jakobsen.
\newblock Monotone schemes.
\newblock In R.~Cont, editor, {\em Encyclopedia of Quantitative Finance}. Wiley
  Online Library, 2010.

\bibitem{JakobsenEstimates}
E.~R. Jakobsen and K.~H. Karlsen.
\newblock Continuous dependence estimates for viscosity solutions of fully
  nonlinear degenerate parabolic equations.
\newblock {\em Journal of Differential Equations}, 183(2):497--525, 2002.

\bibitem{jensen2017notion}
M.~Jensen and I.~Smears.
\newblock On the notion of boundary conditions in comparison principles for
  viscosity solutions.
\newblock {\em arXiv preprint arXiv:1703.07313}, 2017.

\bibitem{krylov1997rate}
N.~V. Krylov.
\newblock On the rate of convergence of finite-difference approximations for
  {B}ellman's equations.
\newblock {\em St. Petersburg Mathematical Journal}, 9(3):639--650, 1998.

\bibitem{krylov2000}
N.~V. Krylov.
\newblock On the rate of convergence of finite-difference approximations for
  {B}ellman's equations with variable coefficients.
\newblock {\em Probability Theory and Related Fields}, 117(1):1--16, 2000.

\bibitem{krylov2015rate}
N.~V. Krylov.
\newblock On the rate of convergence of finite-difference approximations for
  elliptic {I}saacs equations in smooth domains.
\newblock {\em Communications in Partial Differential Equations},
  40(8):1393--1407, 2015.

\bibitem{kushner2001numerical}
H.~J. Kushner and P.~Dupuis.
\newblock {\em {Numerical Methods for Stochastic Control Problems in Continuous
  Time}}, volume~24.
\newblock Springer Science \& Business Media, 2001.

\bibitem{ma2016unconditionally}
K.~Ma and P.~A. Forsyth.
\newblock An unconditionally monotone numerical scheme for the two-factor
  uncertain volatility model.
\newblock {\em IMA Journal of Numerical Analysis}, 37(2):905--944, 2016.

\bibitem{mcshane1934}
E.~J. McShane.
\newblock Extension of range of functions.
\newblock {\em Bulletin of the American Mathematical Society}, 40(12):837--842,
  1934.

\bibitem{oberman2010convergence}
A.~M. Oberman.
\newblock Convergence rates for difference schemes for polyhedral nonlinear
  parabolic equations.
\newblock {\em Journal of Computational Mathematics}, pages 474--488, 2010.

\bibitem{oleinik1973second}
O.~A. Oleinik and E.~V. Radkevich.
\newblock {\em Second Order Equations with Nonnegative Characteristic Form}.
\newblock American Mathematical Society Providence, 1973.

\bibitem{pham2009continuous}
H.~Pham.
\newblock {\em Continuous-time stochastic control and optimization with
  financial applications}, volume~61.
\newblock Springer, 2009.

\bibitem{picarelli2017boundary}
A.~Picarelli, C.~Reisinger, and J.~Rotaetxe~Arto.
\newblock Boundary mesh refinement for semi-{L}agrangian schemes.
\newblock In D.~Kalise, K.~Kunisch, and Z.~Rao, editors, {\em
  {Hamilton-Jacobi-Bellman Equations: Numerical Methods and Applications in
  Optimal Control}}. De Gruyter, 2018.

\bibitem{ReisingerRotaetxe17}
C.~Reisinger and J.~Rotaetxe~Arto.
\newblock Boundary treatment and multigrid preconditioning for
  semi-{L}agrangian schemes applied to {H}amilton--{J}acobi--{B}ellman
  equations.
\newblock {\em Journal of Scientific Computing}, 72(1):198--230, 2017.

\bibitem{safonov}
M.~V. Safonov.
\newblock On the boundary value problems for fully nonlinear elliptic equations
  of second order.
\newblock Technical Report MRR 049--94, The Australian National University,
  Canberra, January 1994.
\newblock Accessed 13 June 2018.

\end{thebibliography}
\bibliographystyle{abbrv}  
\end{document}